\documentclass[english,oneside,a4paper,11pt]{article}
\usepackage{amssymb}
\usepackage{latexsym}
\usepackage[english]{babel}
\usepackage{amsmath}
\usepackage{amsfonts}
\usepackage{amsbsy}
\usepackage{mathrsfs}
\usepackage[latin1]{inputenc}
\usepackage{verbatim}
\usepackage{mathrsfs}
\usepackage{graphicx}
\usepackage{float}
\usepackage{bbm}
\usepackage{lscape}
\usepackage{empheq}
\usepackage[dvipsnames]{xcolor}
\usepackage{array}
\usepackage{a4wide}
\usepackage{enumitem}
\usepackage[dvips,colorlinks,bookmarks,breaklinks]{hyperref}  
\hypersetup{linkcolor=black,citecolor=black,filecolor=black,urlcolor=black}

\usepackage{xcolor}
\usepackage{soul}

\providecommand{\abs}[1]{\left\lvert#1 \right\rvert}

\newtheorem{theorem}{Theorem}

\newtheorem{corollary}{Corollary}

\newenvironment{remark}
{\begin{trivlist}\item[\hskip%
\labelsep{{\it \noindent Remark}}]}{\hfill
\end{trivlist}}

\newenvironment{proof}
{\begin{trivlist}\item[\hskip%
\labelsep{\it \noindent Proof.}]}{\hfill $\square$ \rm
\end{trivlist}}

\newenvironment{example}
{\begin{trivlist}\item[\hskip%
\labelsep{{\sc \noindent Example}}]}{\hfill
\end{trivlist}}

\numberwithin{equation}{section}

\allowdisplaybreaks

\begin{document}
\begin{center}
{\huge {\bf Convergence of series of moments \\ on general exponential inequality}} \\
\vspace{1.0cm}
{\Large João Lita da Silva\footnote{\textit{E-mail address:} \texttt{jfls@fct.unl.pt}; \texttt{joao.lita@gmail.com} (corresponding author)}} \\
\vspace{0.1cm}
\textit{Department of Mathematics and GeoBioTec \\ NOVA School of Sciences and Technology \\
NOVA University of Lisbon \\ Quinta da Torre, 2829-516 Caparica,
Portugal} \\
\vspace{0.75cm}
{\Large Vanda Louren\c{c}o\footnote{\textit{E-mail address:} \texttt{vmml@fct.unl.pt}}} \\
\vspace{0.1cm}
\textit{Department of Mathematics and CMA \\ NOVA School of Sciences and Technology \\
NOVA University of Lisbon \\ Quinta da Torre, 2829-516 Caparica,
Portugal}
\end{center}

\bigskip

\bigskip

\begin{abstract}
\noindent For an array $\left\{X_{n,j}, \, 1 \leqslant j \leqslant k_{n}, n \geqslant 1 \right\}$ of random variables and a sequence $\{c_{n} \}$ of positive numbers, sufficient conditions are given under which, for all $\varepsilon > 0$,
$\sum_{n=1}^{\infty} c_{n} \mathbb{E} \bigg[{\displaystyle \max_{1 \leqslant i \leqslant k_{n}}} \allowbreak \Big\lvert\sum_{j=1}^{i} (X_{n,j} - \mathbb{E} \, X_{n,j} I_{\left\{\lvert X_{n,j} \rvert \leqslant \delta \right\}}) \Big\rvert - \varepsilon \bigg]_{+}^{p} < \infty,$ where $x_{+}$ denotes the positive part of $x$ and $p \geqslant 1$, $\delta > 0$. Our statements are announced in a general setting allowing to conclude the previous convergence for well-known dependent structures. As an application, we study complete consistency and consistency in the $r$th mean of cumulative sum type estimators of the change in the mean of dependent observations.
\end{abstract}

\bigskip

{\small{\textit{Keywords:} maximum of partial sums, convergence of series of moments, dependent random variables, CUSUM-type estimators}}

\bigskip

{\small{\textbf{2010 Mathematics Subject Classification:} 60F15; 62F12}}

\bigskip

\section{Introduction}\label{sec:1}

\indent

Probability inequalities play a central role in proofs of asymptotic results for sums of random variables with maximal probability inequalities being key to the establishment of sharp rates of convergence. The main goal of this paper is to obtain the convergence of series of moments involving the maximum of partial row sums of arrays of random variables under a general exponential-type maximal inequality. Our results encompass some of those of \cite{Chen08}, \cite{Shen16} and \cite{Wang14} as particular cases. Furthermore, our proofs are presented in a more direct and optimised way that enables to weaken the assumptions presented in \cite{Shen16}.

The paper is organized as follows. Section \ref{sec:2} presents our results and proofs. In Section \ref{sec:3}, we consider a (general) change-point model and apply our results to a cumulative sum (CUSUM) estimator of the point of shift in the mean of dependent observations, providing sufficient conditions for its complete consistency and consistency in the $r$th mean. Our theoretical results are illustrated via simulation and at the light of the CUSUM estimator described in Section \ref{sec:3}.

\bigskip

In what follows, $\{k_{n} \}$ indicates a sequence of positive integers such that $k_{n} \rightarrow \infty$ as $n \rightarrow \infty$, and $I_{A}$ stands for the indicator random variable of an event $A$. We additionally write $x \wedge y$ and $x \vee y$ for $\min(x,y)$ and $\max(x,y)$, respectively. Moreover, for each $\ell > 0$, we define the function $g_{\ell}(x) = (x \wedge \ell) \vee (-\ell)$, which describes the truncation at level $\ell$.

\section{Main results}\label{sec:2}

\indent

Considering an array $\left\{X_{n,j}, \, 1 \leqslant j \leqslant k_{n}, n \geqslant 1 \right\}$ of random variables, we admit that, for each $n \geqslant 1$ and all $\ell, \lambda, \eta > 0$, there exist $C_{1}, C_{2} > 0$ non-depending on $n, \ell, \lambda, \eta$ such that
\begin{equation}\label{eq:2.1}
\begin{split}
&\mathbb{P} \left\{\max_{1 \leqslant i \leqslant k_{n}} \abs{\sum_{j=1}^{i} \left[g_{\ell}(X_{n,j}) - \mathbb{E} \, g_{\ell}(X_{n,j}) \right]} \geqslant \lambda \right\} \\
&\quad \leqslant \alpha_{n} \mathbb{P} \left\{\max_{1 \leqslant j \leqslant k_{n}} \abs{\left[g_{\ell}(X_{n,j}) - \mathbb{E} \, g_{\ell}(X_{n,j}) \right]} > C_{1} \eta \right\} + \beta_{n} \left\{\frac{\sum_{j=1}^{k_{n}} \mathbb{E} [g_{\ell}(X_{n,j}) - \mathbb{E}g_{\ell}(X_{n,j})]^{2}}{C_{2} \lambda \eta} \right\}^{\lambda/\eta}
\end{split}
\end{equation}
for some sequences $\{\alpha_{n} \}$, $\{\beta_{n} \}$ of nonnegative numbers. Let us point out that well-known inequalities are contained in \eqref{eq:2.1}, particularly, the ones listed below.

\begin{example}
Consider $s_{n}(\ell) := \sum_{j=1}^{k_{n}} \mathbb{E} [g_{\ell}(X_{n,j}) - \mathbb{E}g_{\ell}(X_{n,j})]^{2} \neq 0$ for all $\ell > 0$ and $n \geqslant 1$. \\

\noindent \textbf{1.} The notion of $m$-negatively associated random variables was introduced in \cite{Hu09} (see Definition 2 of \cite{Hu09}). Let $\left\{X_{n,j}, \, 1 \leqslant j \leqslant k_{n}, n \geqslant 1 \right\}$ be an array of \textit{row-wise $m$-negatively associated random variables}, i.e. for each $n \geqslant 1$, the sequence $\left\{X_{n,j}, \, 1 \leqslant j \leqslant k_{n} \right\}$ of random variables is $m$-negatively associated. According to Property $6$ of \cite{Joag-Dev83}, for all $\ell>0$ and each $n \geqslant 1$, $\left\{g_{\ell}(X_{n,j}), \, 1 \leqslant j \leqslant k_{n} \right\}$ is a sequence of $m$-negatively associated random variables, so that $\left\{g_{\ell}(X_{n,j}), \, 1 \leqslant j \leqslant k_{n}, n \geqslant 1 \right\}$ is an array of row-wise $m$-negatively associated random variables. By Lemma 2.1 and Remark 2.1 of \cite{Wu15}, we have, for all $x,a > 0$ and $n \geqslant 1$,
\begin{equation}\label{eq:2.2}
\begin{split}
&\mathbb{P} \left\{\max_{1 \leqslant i \leqslant k_{n}} \abs{\sum_{j=1}^{i} \left[g_{\ell}(X_{n,j}) - \mathbb{E} \, g_{\ell}(X_{n,j}) \right]} \geqslant x \right\} \\
&\quad \leqslant 2m \mathbb{P} \left\{\max_{1 \leqslant j \leqslant k_{n}} \abs{\left[g_{\ell}(X_{n,j}) - \mathbb{E} \, g_{\ell}(X_{n,j}) \right]} > a \right\} + 8m \left[1 + \frac{3xa}{2m s_{n}(\ell)} \right]^{-x/(12ma)} \\
&\quad \leqslant 2m \mathbb{P} \left\{\max_{1 \leqslant j \leqslant k_{n}} \abs{\left[g_{\ell}(X_{n,j}) - \mathbb{E} \, g_{\ell}(X_{n,j}) \right]} > a \right\} + 8m \left[\frac{2m s_{n}(\ell)}{3xa} \right]^{x/(12ma)}.
\end{split}
\end{equation}
In fact, inequality \eqref{eq:2.2} can be directly obtained from \eqref{eq:2.1} by choosing $\lambda = x$, $\eta = 12m a$, $\alpha_{n} = 2m$, $\beta_{n} = 8m$, with $C_{1} = 1/(12m)$ and $C_{2} = 1/(8 m^{2})$. \\

\noindent \textbf{2.} The concept of negatively superadditive dependent random variables appeared in \cite{Hu00} and it has been employed by some authors since them. If for each $n \geqslant 1$, the sequence $\left\{X_{n,j}, \, 1 \leqslant j \leqslant k_{n} \right\}$ of random variables is negatively superadditive dependent, then the array $\left\{X_{n,j}, \, 1 \leqslant j \leqslant k_{n}, n \geqslant 1 \right\}$ is said to be \textit{row-wise negatively superadditive dependent}. Supposing an array $\left\{X_{n,j}, \, 1 \leqslant j \leqslant k_{n}, n \geqslant 1 \right\}$ of row-wise negatively superadditive dependent random variables and $\ell>0$, we have that, for each $n \geqslant 1$, $\left\{g_{\ell}(X_{n,j}), \, 1 \leqslant j \leqslant k_{n} \right\}$ is still a sequence of negatively superadditive dependent random variables (see P3 in \cite{Hu00}), and whence $\left\{g_{\ell}(X_{n,j}), \, 1 \leqslant j \leqslant k_{n}, n \geqslant 1  \right\}$ is an array of negatively superadditive dependent random variables. Thus, for every $x,y > 0$ and $n \geqslant 1$,
\begin{equation}\label{eq:2.3}
\begin{split}
&\mathbb{P} \left\{\max_{1 \leqslant i \leqslant k_{n}} \abs{\sum_{j=1}^{i} \left[g_{\ell}(X_{n,j}) - \mathbb{E} \, g_{\ell}(X_{n,j}) \right]} \geqslant x \right\} \\
&\quad \leqslant 2 \mathbb{P} \left\{\max_{1 \leqslant j \leqslant k_{n}} \abs{\left[g_{\ell}(X_{n,j}) - \mathbb{E} \, g_{\ell}(X_{n,j}) \right]} > y \right\} + 8 \left[\frac{2 s_{n}(\ell)}{3xy} \right]^{x/(12y)}
\end{split}
\end{equation}
(see Lemma 3.1 of \cite{Wang14}), which follows from \eqref{eq:2.1} by taking $\lambda = x$, $\eta = 12a$, $\alpha_{n} = 2$, $\beta_{n} = 8$; in this case, $C_{1} = 1/12$ and $C_{2} = 1/8$.
\end{example}

Throughout the proofs below, given $t > 0$ and an array of random variables $\{X_{n,j}, \, 1 \leqslant j \leqslant k_{n}, n \geqslant 1 \}$, we denote
\begin{equation}
\begin{gathered}\label{eq:2.4}
X_{n,j}'(t) := X_{n,j} I_{\left\{\abs{X_{n,j}} \leqslant t \right\}}, \\
X_{n,j}''(t) := t I_{\left\{X_{n,j} > t \right\}} - t I_{\left\{X_{n,j} < -t \right\}}, \\
S_{n,i}(t) := \sum_{j=1}^{i} \big(X_{n,j} - \mathbb{E} \, X_{n,j} I_{\left\{\abs{X_{n,j}} \leqslant t \right\}} \big), \quad i = 1,2,\ldots,k_{n}, \\
\Gamma_{n}(t) := \bigcap_{j=1}^{k_{n}} \big\{X_{n,j} = X_{n,j}'(t) \big\}.
\end{gathered}
\end{equation}

Our first general statement is the following.

\begin{theorem}\label{thr:1}
Let $\{c_{n} \}$ be a sequence of positive numbers, and $\left\{X_{n,j}, \, 1 \leqslant j \leqslant k_{n}, n \geqslant 1 \right\}$ be an array of random variables satisfying \eqref{eq:2.1} for some sequences $\{\alpha_{n} \}$, $\{\beta_{n} \}$ of nonnegative numbers. If \\

\noindent \textnormal{(i)} for all $\lambda > 0$, $\sum_{n=1}^{\infty} c_{n} (1 + \alpha_{n}) \sum_{j=1}^{k_{n}} \mathbb{P} \big\{\lvert X_{n,j} \rvert > \lambda \big\} < \infty$, \\

\noindent \textnormal{(ii)} there exist $\delta > 0$ and $q \geqslant 1$ such that
\begin{equation*}
\sum_{n=1}^{\infty} c_{n} \beta_{n} \left(\sum_{j=1}^{k_{n}} \mathbb{P} \big\{\lvert X_{n,j} \rvert > \delta \big\} \right)^{q} < \infty \quad and \quad \sum_{n=1}^{\infty} c_{n} \beta_{n} \left(\sum_{j=1}^{k_{n}} \mathbb{V} X_{n,j} I_{\left\{\lvert X_{n,j} \rvert \leqslant \delta \right\}} \right)^{q} < \infty,
\end{equation*}

\noindent then, for all $\varepsilon > 0$,
\begin{equation}\label{eq:2.5}
\sum_{n=1}^{\infty} c_{n} \mathbb{P} \left\{\max_{1 \leqslant i \leqslant k_{n}} \abs{\sum_{j=1}^{i} (X_{n,j} - \mathbb{E} \, X_{n,j} I_{\left\{\lvert X_{n,j} \rvert \leqslant \delta \right\}})} > \varepsilon \right\} < \infty.
\end{equation}
\end{theorem}

\begin{proof}
Fixing arbitrarily $\varepsilon > 0$ and choosing $\lambda = \varepsilon/2$, $\eta = \varepsilon/(2q)$ there are $C_{1},C_{2}>0$ such that \eqref{eq:2.1} holds with sequences $\{\alpha_{n} \}$, $\{\beta_{n} \}$. Considering $X_{n,j}'(\delta)$, $X_{n,j}''(\delta)$, $\Gamma_{n}(\delta)$ and $S_{n,i}(\delta)$, $i = 1,2,\ldots,k_{n}$ defined in \eqref{eq:2.4}, it follows $X_{n,j}'(\delta) + X_{n,j}''(\delta) = g_{\delta}(X_{n,j})$ and
\begin{align*}
&\sum_{n=1}^{\infty} c_{n} \mathbb{P} \left\{\max_{1 \leqslant i \leqslant k_{n}} \lvert S_{n,i}(\delta) \rvert > \varepsilon \right\} \\
&\quad = \sum_{n=1}^{\infty} c_{n} \mathbb{P} \left[\left\{\max_{1 \leqslant i \leqslant k_{n}} \lvert S_{n,i}(\delta) \rvert > \varepsilon \right\} \cap \Gamma_{n}(\delta) \right] + \sum_{n=1}^{\infty} c_{n} \mathbb{P} \left[\left\{\max_{1 \leqslant i \leqslant k_{n}} \lvert S_{n,i}(\delta) \rvert > \varepsilon \right\} \cap \Gamma_{n}(\delta)^{\complement} \right] \\
&\quad \leqslant \sum_{n=1}^{\infty} c_{n} \mathbb{P} \left\{\max_{1 \leqslant i \leqslant k_{n}} \abs{\sum_{j=1}^{i} \left[X_{n,j}'(\delta) - \mathbb{E} \, X_{n,j}'(\delta) \right]} > \varepsilon \right\} + \sum_{n=1}^{\infty} c_{n} \sum_{j=1}^{k_{n}} \mathbb{P} \big\{\lvert X_{n,j} \rvert > \delta \big\} \\
&\quad = \sum_{n=1}^{\infty} c_{n} \mathbb{P} \left\{\max_{1 \leqslant i \leqslant k_{n}} \abs{\sum_{j=1}^{i} \left[g_{\delta}(X_{n,j}) - \mathbb{E} g_{\delta}(X_{n,j}) + \mathbb{E} X_{n,j}''(\delta) - X_{n,j}''(\delta) \right]} > \varepsilon \right\} \\
&\qquad + \sum_{n=1}^{\infty} c_{n} \sum_{j=1}^{k_{n}} \mathbb{P} \big\{\lvert X_{n,j} \rvert > \delta \big\} \\
&\quad \leqslant \sum_{n=1}^{\infty} c_{n} \mathbb{P} \left\{\max_{1 \leqslant i \leqslant k_{n}} \abs{\sum_{j=1}^{i} \left[g_{\delta}(X_{n,j}) - \mathbb{E} g_{\delta}(X_{n,j}) \right]} > \frac{\varepsilon}{2} \right\} \\
&\qquad + \sum_{n=1}^{\infty} c_{n} \mathbb{P} \left\{\max_{1 \leqslant i \leqslant k_{n}} \abs{\sum_{j=1}^{i} \left[X_{n,j}''(\delta) - \mathbb{E} X_{n,j}''(\delta) \right]} > \frac{\varepsilon}{2} \right\} + \sum_{n=1}^{\infty} c_{n} \sum_{j=1}^{k_{n}} \mathbb{P} \big\{\lvert X_{n,j} \rvert > \delta \big\} \\
&\quad \leqslant \sum_{n=1}^{\infty} c_{n} \mathbb{P} \left\{\max_{1 \leqslant i \leqslant k_{n}} \abs{\sum_{j=1}^{i} \left[g_{\delta}(X_{n,j}) - \mathbb{E} g_{\delta}(X_{n,j}) \right]} > \frac{\varepsilon}{2} \right\} \\
&\qquad + \sum_{n=1}^{\infty} \frac{2 c_{n}}{\varepsilon} \mathbb{E} \left\{\max_{1 \leqslant i \leqslant k_{n}} \abs{\sum_{j=1}^{i} \left[X_{n,j}''(\delta) - \mathbb{E} X_{n,j}''(\delta) \right]} \right\} + \sum_{n=1}^{\infty} c_{n} \sum_{j=1}^{k_{n}} \mathbb{P} \big\{\lvert X_{n,j} \rvert > \delta \big\} \\
&\quad \leqslant \sum_{n=1}^{\infty} c_{n} \mathbb{P} \left\{\max_{1 \leqslant i \leqslant k_{n}} \abs{\sum_{j=1}^{i} \left[g_{\delta}(X_{n,j}) - \mathbb{E} g_{\delta}(X_{n,j}) \right]} > \frac{\varepsilon}{2} \right\} + \left(1 + \frac{4 \delta}{\varepsilon} \right) \sum_{n=1}^{\infty} c_{n} \sum_{j=1}^{k_{n}} \mathbb{P} \big\{\lvert X_{n,j} \rvert > \delta \big\}.
\end{align*}
According to assumption (i), it suffices to prove
\begin{equation*}
\sum_{n=1}^{\infty} c_{n} \mathbb{P} \left\{\max_{1 \leqslant i \leqslant k_{n}} \abs{\sum_{j=1}^{i} \left[g_{\delta}(X_{n,j}) - \mathbb{E} g_{\delta}(X_{n,j}) \right]} > \frac{\varepsilon}{2} \right\} < \infty.
\end{equation*}
Letting $\mathbf{N}$ be the set of positive integers, and
\begin{equation}\label{eq:2.6}
N = N(\delta, \varepsilon,q) = \left\{n \in \mathbf{N}\colon \sum_{j=1}^{k_{n}} \mathbb{P} \left\{\lvert X_{n,j} \rvert > \delta \wedge \frac{\varepsilon C_{1}}{8q} \right\} \geqslant \frac{C_{1} \varepsilon}{16 \delta q} \right\}
\end{equation}
we have $\mathbf{N} = (\mathbf{N} \setminus N) \cup N$. Since (i) implies
\begin{align*}
&\sum_{n \in N} c_{n} \mathbb{P} \left\{\max_{1 \leqslant i \leqslant k_{n}} \abs{\sum_{j=1}^{i} \left[g_{\delta}(X_{n,j}) - \mathbb{E} g_{\delta}(X_{n,j}) \right]} > \frac{\varepsilon}{2} \right\} \leqslant \sum_{n \in N} c_{n} \\
& \quad \leqslant \sum_{n \in N} \frac{16 \delta q}{C_{1} \varepsilon} c_{n} \sum_{j=1}^{k_{n}} \mathbb{P} \left\{\lvert X_{n,j} \rvert > \delta \wedge \frac{C_{1} \varepsilon q}{8q} \right\} \leqslant \frac{16 \delta q}{C_{1} \varepsilon} \sum_{n=1}^{\infty} c_{n} \sum_{j=1}^{k_{n}} \mathbb{P} \left\{\lvert X_{n,j} \rvert > \delta \wedge \frac{C_{1} \varepsilon}{8q} \right\} < \infty,
\end{align*}
one only needs to show that
\begin{equation}\label{eq:2.7}
\sum_{n \in \mathbf{N} \setminus N} c_{n} \mathbb{P} \left\{\max_{1 \leqslant i \leqslant k_{n}} \abs{\sum_{j=1}^{i} \left[g_{\delta}(X_{n,j}) - \mathbb{E} g_{\delta}(X_{n,j}) \right]} > \frac{\varepsilon}{2} \right\} < \infty.
\end{equation}
From \eqref{eq:2.1}, we obtain
\begin{align*}
&\sum_{n \in \mathbf{N} \setminus N} c_{n} \mathbb{P} \left\{\max_{1 \leqslant i \leqslant k_{n}} \abs{\sum_{j=1}^{i} \left[g_{\delta}(X_{n,j}) - \mathbb{E} g_{\delta}(X_{n,j}) \right]} > \frac{\varepsilon}{2} \right\} \\
&\quad \leqslant \sum_{n \in \mathbf{N} \setminus N} c_{n} \alpha_{n} \mathbb{P} \left\{\max_{1 \leqslant j \leqslant k_{n}} \abs{g_{\delta}(X_{n,j}) - \mathbb{E} g_{\delta}(X_{n,j})} > \frac{C_{1} \varepsilon}{2q} \right\} + \sum_{n \in \mathbf{N} \setminus N} c_{n} \beta_{n} \left[\frac{4q s_{n}(\delta)}{C_{2} \varepsilon^{2}} \right]^{q}.
\end{align*}
Since for every $n \in \mathbf{N} \setminus N$,
\begin{align*}
&\max_{1 \leqslant j \leqslant k_{n}} \lvert \mathbb{E} g_{\delta}(X_{n,j}) \rvert \\
&\quad \leqslant \max_{1 \leqslant j \leqslant k_{n}} \left[\mathbb{E} \lvert X_{n,j} \rvert I_{\left\{\lvert X_{n,j} \rvert \leqslant \delta \right\}} + \delta \mathbb{P} \big\{\lvert X_{n,j} \rvert > \delta \big\} \right] \\
&\quad \leqslant \max_{1 \leqslant j \leqslant k_{n}} \left[\mathbb{E} \lvert X_{n,j} \rvert I_{\left\{\lvert X_{n,j} \rvert \leqslant \delta \wedge \frac{C_{1} \varepsilon}{8q} \right\}} + \mathbb{E} \lvert X_{n,j} \rvert I_{\left\{\delta \wedge \frac{C_{1} \varepsilon}{8q} < \lvert X_{n,j} \rvert \leqslant \delta \right\}} + \delta \mathbb{P} \left\{\lvert X_{n,j} \rvert > \delta \right\} \right] \\
&\quad \leqslant \delta \wedge \frac{C_{1} \varepsilon}{8q} + \max_{1 \leqslant j \leqslant k_{n}} \left[\delta \mathbb{P} \left\{\lvert X_{n,j} \rvert > \delta \wedge \frac{C_{1} \varepsilon}{8q} \right\} + \delta \mathbb{P} \big\{\lvert X_{n,j} \rvert > \delta \big\} \right] \\
&\quad \leqslant \frac{C_{1} \varepsilon}{8q} + \max_{1 \leqslant j \leqslant k_{n}} \left[2 \delta \mathbb{P} \left\{\lvert X_{n,j} \rvert > \delta \wedge \frac{C_{1} \varepsilon}{8q} \right\} \right] \\
&\quad \leqslant \frac{C_{1} \varepsilon}{8q} + 2 \delta \sum_{j=1}^{k_{n}} \mathbb{P} \left\{\lvert X_{n,j} \rvert > \delta \wedge \frac{C_{1} \varepsilon}{8q} \right\} \\
&\quad < \frac{C_{1} \varepsilon}{4q}
\end{align*}
we get
\begin{equation}\label{eq:2.8}
\begin{split}
&\sum_{n \in \mathbf{N} \setminus N} c_{n} \alpha_{n} \mathbb{P} \left\{\max_{1 \leqslant j \leqslant k_{n}} \abs{g_{\delta}(X_{n,j}) - \mathbb{E} g_{\delta}(X_{n,j})} > \frac{C_{1} \varepsilon}{2q} \right\} \\
&\quad \leqslant \sum_{n \in \mathbf{N} \setminus N} c_{n} \alpha_{n} \mathbb{P} \left\{\max_{1 \leqslant j \leqslant k_{n}} \abs{g_{\delta}(X_{n,j})} > \frac{C_{1} \varepsilon}{4q} \right\} \\
&\quad \leqslant \sum_{n \in \mathbf{N} \setminus N} c_{n} \alpha_{n} \sum_{j=1}^{k_{n}} \mathbb{P} \left\{\abs{g_{\delta}(X_{n,j})} > \frac{C_{1} \varepsilon}{4q} \right\} \\
&\quad \leqslant \sum_{n \in \mathbf{N} \setminus N} c_{n} \alpha_{n} \sum_{j=1}^{k_{n}} \mathbb{P} \left\{\lvert X_{n,j} \rvert > \frac{C_{1} \varepsilon}{4q} \right\} \\
&\quad \leqslant \sum_{n=1}^{\infty} c_{n} \alpha_{n} \sum_{j=1}^{k_{n}} \mathbb{P} \left\{\lvert X_{n,j} \rvert > \frac{C_{1} \varepsilon}{4q} \right\} \\
&\quad < \infty
\end{split}
\end{equation}
by virtue of assumption (i). According to Cauchy-Schwarz inequality (see \cite{Mukhopadhyay00}, page $150$) and the elementary inequality, $2uv \leqslant u^{2} + v^{2}$ for all real numbers $u,v$, we have
\begin{align*}
s_{n}(\delta) &:= \sum_{j=1}^{k_{n}} \mathbb{E} [g_{\delta}(X_{n,j}) - \mathbb{E}g_{\delta}(X_{n,j})]^{2} \\
&= \sum_{j=1}^{k_{n}} \mathbb{V} [g_{\delta}(X_{n,j})] \\
&= \sum_{j=1}^{k_{n}} \mathbb{V} X_{n,j}'(\delta) + \sum_{j=1}^{k_{n}} \mathbb{V} X_{n,j}''(\delta) + 2\sum_{j=1}^{k_{n}} \mathrm{Cov}\left[X_{n,j}'(\delta),X_{n,j}''(\delta) \right] \\
&\leqslant \sum_{j=1}^{k_{n}} \mathbb{V} X_{n,j}'(\delta) + \sum_{j=1}^{k_{n}} \mathbb{V} X_{n,j}''(\delta) + 2\sum_{j=1}^{k_{n}} \left[\mathbb{V} X_{n,j}'(\delta) \right]^{1/2} \left[\mathbb{V} X_{n,j}''(\delta) \right]^{1/2} \\
&\leqslant 2\sum_{j=1}^{k_{n}} \mathbb{V} X_{n,j}'(\delta) + 2 \sum_{j=1}^{k_{n}} \mathbb{V} X_{n,j}''(\delta) \\
&\leqslant 2\sum_{j=1}^{k_{n}} \mathbb{V} X_{n,j}'(\delta) + 2 \delta^{2} \sum_{j=1}^{k_{n}} \mathbb{P} \{\lvert X_{n,j} \rvert > \delta \}
\end{align*}
where the last inequality follows from $\mathbb{V} X_{n,j}''(\delta) \leqslant \mathbb{E} [X_{n,j}''(\delta)]^{2} = \delta^{2} \mathbb{P} \{\lvert X_{n,j} \rvert > \delta \}$. By assumption (ii) and the inequality $(u + v)^{q} \leqslant 2^{q - 1} (u^{q} + v^{q})$ for any $u,v \geqslant 0$, $q \geqslant 1$, we obtain
\begin{equation}\label{eq:2.9}
\begin{split}
&\sum_{n \in \mathbf{N} \setminus N} c_{n} \beta_{n} \left[\frac{4q s_{n}(\delta)}{C_{2} \varepsilon^{2}} \right]^{q} \\
&\quad \leqslant \sum_{n \in \mathbf{N} \setminus N} c_{n} \beta_{n} \frac{2^{3q-1}q^{q}}{C_{2}^{q} \varepsilon^{2q}} \left\{\left[2 \sum_{j=1}^{k_{n}} \mathbb{V} X_{n,j}'(\delta) \right]^{q} + \left(2 \delta^{2} \sum_{j=1}^{k_{n}} \mathbb{P} \{\lvert X_{n,j} \rvert > \delta \} \right)^{q} \right\} \\
&\quad \leqslant \frac{2^{4q-1}q^{q}}{C_{2}^{q} \varepsilon^{2q}} \sum_{n=1}^{\infty} c_{n}  \beta_{n} \left(\sum_{j=1}^{k_{n}} \mathbb{V} X_{n,j} I_{\left\{\lvert X_{n,j} \rvert \leqslant \delta \right\}} \right)^{q} + \frac{2^{4q-1} q^{q}\delta^{2q}}{C_{2}^{q} \varepsilon^{2q}} \sum_{n=1}^{\infty} c_{n}  \beta_{n} \left(\sum_{j=1}^{k_{n}} \mathbb{P} \big\{\lvert X_{n,j} \rvert > \delta \big\} \right)^{q} \\
&\quad < \infty.
\end{split}
\end{equation}
The thesis is established.
\end{proof}

Inequality \eqref{eq:2.1} with $\{\alpha_{n} \}$ and $\{\beta_{n} \}$ bounded sequences allows us to improve assumption (ii) of Theorem~\ref{thr:1}.

\begin{theorem}\label{thr:2}
Let $\{c_{n} \}$ be a sequence of positive numbers, and $\left\{X_{n,j}, \, 1 \leqslant j \leqslant k_{n}, n \geqslant 1 \right\}$ be an array of random variables satisfying \eqref{eq:2.1} for sequences $\{\alpha_{n} \}$, $\{\beta_{n} \}$ of nonnegative numbers such that $\alpha_{n} = O(1) = \beta_{n}$ as $n \rightarrow \infty$. If \\

\noindent \textnormal{(i')} for all $\lambda > 0$, $\sum_{n=1}^{\infty} c_{n} \sum_{j=1}^{k_{n}} \mathbb{P} \big\{\lvert X_{n,j} \rvert > \lambda \big\} < \infty$, \\

\noindent \textnormal{(ii')} there exist $\delta > 0$ and $q \geqslant 1$ such that $\sum_{n=1}^{\infty} c_{n} \left(\sum_{j=1}^{k_{n}} \mathbb{V} X_{n,j} I_{\left\{\lvert X_{n,j} \rvert \leqslant \delta \right\}} \right)^{q} < \infty$, \\

\noindent then, for all $\varepsilon > 0$, \eqref{eq:2.5} holds.
\end{theorem}

\begin{proof}
The proof follows exactly the same steps of the preceding one, except \eqref{eq:2.9}. From $\alpha_{n} = O(1) = \beta_{n}$, $n \rightarrow \infty$ there is $C > 0$ (non-depending on $n$) such that $\alpha_{n}, \beta_{n} \leqslant C$ for all $n \geqslant 1$ and instead of \eqref{eq:2.9}, we have
\begin{align*}
&\sum_{n \in \mathbf{N} \setminus N} c_{n} \beta_{n} \left[\frac{4q s_{n}(\delta)}{C_{2} \varepsilon^{2}} \right]^{q} \\
&\quad \leqslant \sum_{n \in \mathbf{N} \setminus N} c_{n} \frac{2^{3q-1}q^{q} C}{C_{2}^{q} \varepsilon^{2q}} \left\{\left[2 \sum_{j=1}^{k_{n}} \mathbb{V} X_{n,j}'(\delta) \right]^{q} + \left(2 \delta^{2} \sum_{j=1}^{k_{n}} \mathbb{P} \{\lvert X_{n,j} \rvert > \delta \} \right)^{q} \right\} \\
&\quad \leqslant \sum_{n=1}^{\infty} c_{n} \frac{2^{4q-1}q^{q} C}{C_{2}^{q} \varepsilon^{2q}} \left(\sum_{j=1}^{k_{n}} \mathbb{V} X_{n,j} I_{\left\{\lvert X_{n,j} \rvert \leqslant \delta \right\}} \right)^{q} + \sum_{n \in \mathbf{N} \setminus N} c_{n} \frac{\delta^{q} C_{1}^{q} C}{2 C_{2}^{q} \varepsilon^{q}} \left(\frac{16 \delta q}{C_{1} \varepsilon} \sum_{j=1}^{k_{n}} \mathbb{P} \big\{\lvert X_{n,j} \rvert > \delta \big\} \right)^{q} \\
&\quad \leqslant \frac{2^{4q-1} q^{q}C}{C_{2}^{q} \varepsilon^{2q}} \sum_{n=1}^{\infty} c_{n} \left(\sum_{j=1}^{k_{n}} \mathbb{V} X_{n,j} I_{\left\{\lvert X_{n,j} \rvert \leqslant \delta \right\}} \right)^{q} + \sum_{n \in \mathbf{N} \setminus N} c_{n} \frac{8 q \delta^{q+1} C_{1}^{q-1} C}{C_{2}^{q} \varepsilon^{q+1}} \sum_{j=1}^{k_{n}} \mathbb{P} \big\{\lvert X_{n,j} \rvert > \delta \big\} \\
&\quad \leqslant \frac{2^{4q-1} q^{q}C}{C_{2}^{q} \varepsilon^{2q}} \sum_{n=1}^{\infty} c_{n} \left(\sum_{j=1}^{k_{n}} \mathbb{V} X_{n,j} I_{\left\{\lvert X_{n,j} \rvert \leqslant \delta \right\}} \right)^{q} + \frac{8 q \delta^{q+1} C_{1}^{q-1} C}{C_{2}^{q} \varepsilon^{q+1}} \sum_{n=1}^{\infty} c_{n} \sum_{j=1}^{k_{n}} \mathbb{P} \big\{\lvert X_{n,j} \rvert > \delta \big\} \\
&\quad < \infty
\end{align*}
because for each $n \in \mathbf{N} \setminus N$, $\sum_{j=1}^{k_{n}} \mathbb{P} \big\{\lvert X_{n,j} \rvert > \delta \big\} \leqslant \sum_{j=1}^{k_{n}} \mathbb{P} \big\{\lvert X_{n,j} \rvert > \delta \wedge C_{1} \varepsilon/(8q) \big\} < C_{1} \varepsilon/(16 \delta q)$. The proof is complete.
\end{proof}

\begin{remark}
Let us observe that Theorem 1 in \cite{Chen08} or Theorem 1 in \cite{Hu09} can be both obtained from Theorem~\ref{thr:2} by using inequality \eqref{eq:2.2} (particularly with $m = 1$ in the former). Further, by employing inequality \eqref{eq:2.3}, Theorem 3.3 of \cite{Wang14} is also a consequence of Theorem~\ref{thr:2}.
\end{remark}

The next result presents sufficient conditions under which a series of moments involving the maximum of partial row sums of arrays of random variables converges.

\begin{theorem}\label{thr:3}
Let $p \geqslant 1$, $\{c_{n} \}$ be a sequence of positive numbers, and $\left\{X_{n,j}, \, 1 \leqslant j \leqslant k_{n}, n \geqslant 1 \right\}$ be an array of random variables verifying \eqref{eq:2.1} for some sequences $\{\alpha_{n} \}$, $\{\beta_{n} \}$ of nonnegative numbers. If there exists a constant $\delta > 0$ such that \\

\noindent \textnormal{(a)} for any $\lambda > 0$,
\begin{equation*}
\sum_{n=1}^{\infty} c_{n} \mathbb{P} \left\{{\displaystyle \max_{1 \leqslant i \leqslant k_{n}}} \abs{\sum_{j=1}^{i} \left(X_{n,j} - \mathbb{E} X_{n,j}I_{\left\{\lvert X_{n,j} \rvert \leqslant \delta \right\}} \right)} > \lambda \right\} < \infty,
\end{equation*}

\noindent \textnormal{(b)} ${\displaystyle \max_{1 \leqslant j \leqslant k_{n}}} \mathbb{P} \left\{\lvert X_{n,j} \rvert > \delta \right\} = o(1)$ and $\sum_{j=1}^{k_{n}} \mathbb{P} \left\{\lvert X_{n,j} \rvert > \delta \right\} = O(1)$ as $n \rightarrow \infty$, \\

\noindent \textnormal{(c)} $\sum_{n=1}^{\infty} \sum_{j=1}^{k_{n}} c_{n} (1 + \alpha_{n}) \int_{\delta^{p}}^{\infty} \mathbb{P} \left\{\lvert X_{n,j} \rvert^{p} > t \right\} \mathrm{d}t < \infty$, \\

\noindent \textnormal{(d)} for some $q > p$, $\sum_{n=1}^{\infty} c_{n} \beta_{n} \int_{\delta^{p}}^{\infty} \left(\sum_{j=1}^{k_{n}} \mathbb{P} \big\{\lvert X_{n,j} \rvert^{p} > t \big\} \mathrm{d}t \right)^{q} < \infty$,
\begin{equation*}
\sum_{n=1}^{\infty} c_{n} \beta_{n} \left(\sum_{j=1}^{k_{n}} \mathbb{V} X_{n,j} I_{\left\{\lvert X_{n,j} \rvert \leqslant \delta \right\}} \right)^{q} < \infty \quad and \quad \sum_{n=1}^{\infty} c_{n} \beta_{n} \left(\sum_{j=1}^{k_{n}} \mathbb{E} \lvert X_{n,j} \rvert I_{\left\{\lvert X_{n,j} \rvert > \delta \right\}} \right)^{q} < \infty,
\end{equation*}

\noindent then for all $\varepsilon > 0$,
\begin{equation}\label{eq:2.10}
\sum_{n=1}^{\infty} c_{n} \mathbb{E} \left[\max_{1 \leqslant i \leqslant k_{n}} \abs{\sum_{j=1}^{i} (X_{n,j} - \mathbb{E} \, X_{n,j} I_{\left\{\lvert X_{n,j} \rvert \leqslant \delta \right\}})} - \varepsilon \right]_{+}^{p} < \infty.
\end{equation}
\end{theorem}

\begin{proof}
From assumption (b), there is a positive constant $C_{\delta}$ such that
\begin{equation*}
\sum_{j=1}^{k_{n}} \mathbb{P} \big\{\lvert X_{n,j} \rvert > \delta \big\} \leqslant C_{\delta}
\end{equation*}
for all $n$. Consider $X_{n,j}'(t^{1/p})$, $X_{n,j}''(t^{1/p})$, $\Gamma_{n}(t^{1/p})$, $S_{n,i}(\delta) := \sum_{j=1}^{i} \big(X_{n,j} - \mathbb{E} \, X_{n,j} I_{\left\{\lvert X_{n,j} \rvert \leqslant \delta \right\}} \big)$, $i = 1,2,\ldots,k_{n}$ given by \eqref{eq:2.4}, fix arbitrarily $\varepsilon > 0$ and let $\rho$ be a positive number that will be determined later. Hence,
\begin{align*}
\sum_{n=1}^{\infty} c_{n} \mathbb{E} \left(\max_{1 \leqslant i \leqslant k_{n}} \lvert S_{n,i}(\delta) \rvert - \varepsilon \right)_{+}^{p} &= \sum_{n=1}^{\infty} c_{n} \int_{0}^{\infty} \mathbb{P} \left\{\max_{1 \leqslant i \leqslant k_{n}} \lvert S_{n,i}(\delta) \rvert > \varepsilon + u^{1/p} \right\} \mathrm{d}u \\
&= (2C_{\delta})^{p} \sum_{n=1}^{\infty} c_{n} \int_{0}^{\infty} \mathbb{P} \left\{\max_{1 \leqslant i \leqslant k_{n}} \lvert S_{n,i}(\delta) \rvert > \varepsilon + 2C_{\delta}t^{1/p} \right\} \mathrm{d}t \\
& \leqslant (2C_{\delta})^{p} \left(\frac{\delta}{\rho} \right)^{p} \vee \delta^{p} \sum_{n=1}^{\infty} c_{n} \mathbb{P} \left\{\max_{1 \leqslant i \leqslant k_{n}} \lvert S_{n,i}(\delta) \rvert > \varepsilon \right\} \\
&\quad + (2C_{\delta})^{p} \sum_{n=1}^{\infty} c_{n} \int_{(\delta/\rho)^{p} \vee \delta^{p}}^{\infty} \mathbb{P} \left\{\max_{1 \leqslant i \leqslant k_{n}} \lvert S_{n,i}(\delta) \rvert > 2C_{\delta}t^{1/p} \right\} \mathrm{d}t
\end{align*}
and by assumption (a), \eqref{eq:2.10} holds if
\begin{equation*}
\sum_{n=1}^{\infty} c_{n} \int_{(\delta/\rho)^{p} \vee \delta^{p}}^{\infty} \mathbb{P} \left\{\max_{1 \leqslant i \leqslant k_{n}} \lvert S_{n,i}(\delta) \rvert > 2C_{\delta} t^{1/p} \right\} \mathrm{d}t < \infty.
\end{equation*}
Thus,
\begin{align*}
&\sum_{n=1}^{\infty} c_{n} \int_{(\delta/\rho)^{p} \vee \delta^{p}}^{\infty} \mathbb{P} \left\{\max_{1 \leqslant i \leqslant k_{n}} \lvert S_{n,i}(\delta) \rvert > 2C_{\delta}t^{1/p} \right\} \mathrm{d}t \\
&\quad = \sum_{n=1}^{\infty} c_{n} \int_{(\delta/\rho)^{p} \vee \delta^{p}}^{\infty} \mathbb{P} \left[\left\{\max_{1 \leqslant i \leqslant k_{n}} \lvert S_{n,i}(\delta) \rvert > 2C_{\delta}t^{1/p} \right\} \cap \Gamma_{n}\big(t^{1/p} \big) \right] \mathrm{d}t \\
&\qquad + \sum_{n=1}^{\infty} c_{n} \int_{(\delta/\rho)^{p} \vee \delta^{p}}^{\infty} \mathbb{P} \left[\left\{\max_{1 \leqslant i \leqslant k_{n}} \lvert S_{n,i}(\delta) \rvert > 2C_{\delta}t^{1/p} \right\} \cap \Gamma_{n}\big(t^{1/p} \big)^{\complement} \right] \mathrm{d}t \\
&\quad \leqslant \sum_{n=1}^{\infty} c_{n} \int_{(\delta/\rho)^{p} \vee \delta^{p}}^{\infty} \mathbb{P} \left\{\max_{1 \leqslant i \leqslant k_{n}} \abs{\sum_{j=1}^{i} \left(X_{n,j} I_{\left\{\lvert X_{n,j} \rvert \leqslant t^{1/p} \right\}} - \mathbb{E} \, X_{n,j} I_{\left\{\lvert X_{n,j} \rvert \leqslant \delta \right\}} \right)} > 2C_{\delta}t^{1/p} \right\} \mathrm{d}t \\
& \qquad + \sum_{n=1}^{\infty} c_{n} \int_{\delta^{p}}^{\infty} \sum_{j=1}^{k_{n}} \mathbb{P} \big\{\lvert X_{n,j} \rvert > t^{1/p} \big\} \mathrm{d}t.
\end{align*}
According to (c), it suffices to prove
\begin{equation}\label{eq:2.11}
\sum_{n=1}^{\infty} c_{n} \int_{(\delta/\rho)^{p} \vee \delta^{p}}^{\infty} \mathbb{P} \left\{\max_{1 \leqslant i \leqslant k_{n}} \abs{\sum_{j=1}^{i} \left(X_{n,j} I_{\left\{\lvert X_{n,j} \rvert \leqslant t^{1/p} \right\}} - \mathbb{E} \, X_{n,j} I_{\left\{\lvert X_{n,j} \rvert \leqslant \delta \right\}} \right)} > 2C_{\delta}t^{1/p} \right\} \mathrm{d}t < \infty.
\end{equation}
Since
\begin{align*}
\max_{1 \leqslant i \leqslant k_{n}} \abs{\sum_{j=1}^{i} \mathbb{E} \, X_{n,j} I_{\left\{\delta < \lvert X_{n,j} \rvert \leqslant t^{1/p} \right\}}} &\leqslant \sum_{j=1}^{k_{n}} \mathbb{E} \lvert X_{n,j} \rvert I_{\left\{\delta < \lvert X_{n,j} \rvert \leqslant t^{1/p} \right\}} \\
&\leqslant t^{1/p} \sum_{j=1}^{k_{n}} \mathbb{P} \big\{\lvert X_{n,j} \rvert > \delta \big\} \\
&\leqslant C_{\delta} t^{1/p}
\end{align*}
and
\begin{align*}
&\mathbb{P} \left\{\max_{1 \leqslant i \leqslant k_{n}} \abs{\sum_{j=1}^{i} \left[\mathbb{E} X_{n,j}''\big(t^{1/p} \big) - X_{n,j}''\big(t^{1/p} \big) \right]} > \frac{C_{\delta} t^{1/p}}{2} \right\} \\
&\quad \leqslant \frac{2}{C_{\delta} t^{1/p}} \mathbb{E} \left\{\max_{1 \leqslant i \leqslant k_{n}} \abs{\sum_{j=1}^{i}\left[\mathbb{E} X_{n,j}''\big(t^{1/p} \big) - X_{n,j}''\big(t^{1/p} \big) \right]} \right\} \\
&\quad \leqslant \frac{4}{C_{\delta} t^{1/p}} \sum_{j=1}^{k_{n}} \mathbb{E} \big\lvert X_{n,j}''\big(t^{1/p} \big) \big\rvert \\
&\quad = \frac{4}{C_{\delta}} \sum_{j=1}^{k_{n}} \mathbb{P} \big\{\lvert X_{n,j} \rvert > t^{1/p} \big\}
\end{align*}
we obtain, for each $t \geqslant (\delta/\rho)^{p} \vee \delta^{p}$,
\begin{align*}
&\mathbb{P} \left\{\max_{1 \leqslant i \leqslant k_{n}} \abs{\sum_{j=1}^{i} \left(X_{n,j} I_{\left\{\lvert X_{n,j} \rvert \leqslant t^{1/p} \right\}} - \mathbb{E} \, X_{n,j} I_{\left\{\lvert X_{n,j} \rvert \leqslant \delta \right\}} \right)} > 2C_{\delta}t^{1/p} \right\} \\
&\quad = \mathbb{P} \left\{\max_{1 \leqslant i \leqslant k_{n}} \abs{\sum_{j=1}^{i} \left[X_{n,j}'\big(t^{1/p} \big) - \mathbb{E} X_{n,j}'\big(t^{1/p} \big) + \mathbb{E} \, X_{n,j} I_{\left\{\delta < \lvert X_{n,j} \rvert \leqslant t^{1/p} \right\}} \right]} > 2C_{\delta}t^{1/p} \right\} \\
&\quad \leqslant \mathbb{P} \left\{\max_{1 \leqslant i \leqslant k_{n}} \abs{\sum_{j=1}^{i} \left[X_{n,j}'\big(t^{1/p} \big) - \mathbb{E} X_{n,j}'\big(t^{1/p} \big) \right]} + \max_{1 \leqslant i \leqslant k_{n}} \abs{\sum_{j=1}^{i} \mathbb{E} \, X_{n,j} I_{\left\{\delta < \lvert X_{n,j} \rvert \leqslant t^{1/p} \right\}}} > 2C_{\delta}t^{1/p} \right\} \\
&\quad \leqslant \mathbb{P} \left\{\max_{1 \leqslant i \leqslant k_{n}} \abs{\sum_{j=1}^{i} \left[X_{n,j}'\big(t^{1/p} \big) - \mathbb{E} X_{n,j}'\big(t^{1/p} \big) \right]} > C_{\delta}t^{1/p} \right\} \\
& \quad \leqslant \mathbb{P} \left\{\max_{1 \leqslant i \leqslant k_{n}} \abs{\sum_{j=1}^{i} \left[g_{t^{1/p}}(X_{n,j}) - \mathbb{E} g_{t^{1/p}}(X_{n,j}) \right]} > \frac{C_{\delta}t^{1/p}}{2} \right\} \\
&\qquad + \mathbb{P} \left\{\max_{1 \leqslant i \leqslant k_{n}} \abs{\sum_{j=1}^{i} \left[\mathbb{E} X_{n,j}''\big(t^{1/p} \big) - X_{n,j}''\big(t^{1/p} \big) \right]} > \frac{C_{\delta}t^{1/p}}{2} \right\} \\
& \quad \leqslant \mathbb{P} \left\{\max_{1 \leqslant i \leqslant k_{n}} \abs{\sum_{j=1}^{i} \left[g_{t^{1/p}}(X_{n,j}) - \mathbb{E} g_{t^{1/p}}(X_{n,j}) \right]} > \frac{C_{\delta} t^{1/p}}{2} \right\} + \frac{4}{C_{\delta}} \sum_{j=1}^{k_{n}} \mathbb{P} \big\{\lvert X_{n,j} \rvert > t^{1/p} \big\}
\end{align*}
Thus, \eqref{eq:2.11} holds if
\begin{equation}\label{eq:2.12}
\sum_{n=1}^{\infty} c_{n} \int_{(\delta/\rho)^{p} \vee \delta^{p}}^{\infty} \mathbb{P} \left\{\max_{1 \leqslant i \leqslant k_{n}} \abs{\sum_{j=1}^{i} \left[g_{t^{1/p}}(X_{n,j}) - \mathbb{E} g_{t^{1/p}}(X_{n,j}) \right]} > \frac{C_{\delta} t^{1/p}}{2} \right\} < \infty
\end{equation}
and
\begin{equation}\label{eq:2.13}
\sum_{n=1}^{\infty} c_{n} \int_{(\delta/\rho)^{p} \vee \delta^{p}}^{\infty} \sum_{j=1}^{k_{n}} \mathbb{P} \big\{\lvert X_{n,j} \rvert > t^{1/p} \big\} < \infty.
\end{equation}
Since \eqref{eq:2.13} follows from assumption (c), it remains to show \eqref{eq:2.12}. By assumption (b), there is $n_{0} = n_{0}(\delta,\rho)$ such that
\begin{equation}\label{eq:2.14}
\max_{1 \leqslant j \leqslant k_{n}} \mathbb{P} \left\{\lvert X_{n,j} \rvert > \delta \right\} < \frac{\rho}{2}, \quad \forall n \geqslant n_{0}
\end{equation}
and whence, for any $t \geqslant (\delta/\rho)^{p} \vee \delta^{p}$ and $n \geqslant n_{0}$, we have
\begin{equation}\label{eq:2.15}
\begin{split}
&\max_{1 \leqslant j \leqslant k_{n}} \lvert \mathbb{E} g_{t^{1/p}}(X_{n,j}) \rvert \\
&\quad \leqslant \max_{1 \leqslant j \leqslant k_{n}} \left[\mathbb{E} \lvert X_{n,j} \rvert I_{\left\{\lvert X_{n,j} \rvert \leqslant t^{1/p} \right\}} + t^{1/p} \mathbb{P} \big\{\lvert X_{n,j} \rvert > t^{1/p} \big\} \right] \\
&\quad \leqslant \max_{1 \leqslant j \leqslant k_{n}} \left[\mathbb{E} \lvert X_{n,j} \rvert I_{\left\{\lvert X_{n,j} \rvert \leqslant \delta \right\}} + \mathbb{E} \lvert X_{n,j} \rvert I_{\left\{\delta < \lvert X_{n,j} \rvert \leqslant t^{1/p} \right\}} + t^{1/p} \mathbb{P} \big\{\lvert X_{n,j} \rvert > t^{1/p} \big\} \right] \\
&\quad \leqslant \delta + \max_{1 \leqslant j \leqslant k_{n}} \left[t^{1/p} \mathbb{P} \big\{\lvert X_{n,j} \rvert > \delta \big\} + t^{1/p} \mathbb{P} \big\{\lvert X_{n,j} \rvert > t^{1/p} \big\} \right] \\
&\quad \leqslant \rho t^{1/p} + \max_{1 \leqslant j \leqslant k_{n}} \left[2 t^{1/p} \mathbb{P} \big\{\lvert X_{n,j} \rvert > \delta \big\} \right] \\
&\quad < 2\rho t^{1/p}.
\end{split}
\end{equation}
By taking $\lambda = C_{\delta} t^{1/p}/2$ and $\eta = C_{\delta}t^{1/p}/(2q)$, there are $C_{1},C_{2} > 0$ such that \eqref{eq:2.1} holds with the sequences $\{\alpha_{n} \}$, $\{\beta_{n} \}$. Putting $\rho = C_{1}C_{\delta}/(6q)$, estimates \eqref{eq:2.15} and
\begin{align*}
s_{n}(t^{1/p})& = \sum_{j=1}^{k_{n}} \mathbb{E} [g_{t^{1/p}}(X_{n,j}) - \mathbb{E}g_{t^{1/p}}(X_{n,j})]^{2} \\
&= \sum_{j=1}^{k_{n}} \mathbb{V} [g_{t^{1/p}}(X_{n,j})] \\
&\leqslant 2\sum_{j=1}^{k_{n}} \mathbb{V} X_{n,j}'(t^{1/p}) + 2\sum_{j=1}^{k_{n}} \mathbb{V} X_{n,j}''(t^{1/p}) \\
&\leqslant 2\sum_{j=1}^{k_{n}} \mathbb{V} \left[X_{n,j}'(\delta) + X_{n,j}I_{\left\{\delta < \lvert X_{n,j} \rvert \leqslant t^{1/p} \right\}} \right]
+ 2 \sum_{j=1}^{k_{n}} \mathbb{E} \left[X_{n,j}''(t^{1/p}) \right]^{2} \\
&\leqslant 4\sum_{j=1}^{k_{n}} \mathbb{V} X_{n,j}'(\delta) + 4\sum_{j=1}^{k_{n}} \mathbb{V} X_{n,j}I_{\left\{\delta < \lvert X_{n,j} \rvert \leqslant t^{1/p} \right\}}
+ 2 t^{2/p} \sum_{j=1}^{k_{n}} \mathbb{P} \big\{\lvert X_{n,j} \rvert > t^{1/p} \big\} \\
&\leqslant 4\sum_{j=1}^{k_{n}} \mathbb{V} X_{n,j} I_{\left\{\abs{X_{n,j}} \leqslant \delta \right\}} + 4\sum_{j=1}^{k_{n}} \mathbb{E} X_{n,j}^{2} I_{\left\{\delta < \lvert X_{n,j} \rvert \leqslant t^{1/p} \right\}} + 2 t^{2/p} \sum_{j=1}^{k_{n}} \mathbb{P} \big\{\lvert X_{n,j} \rvert > t^{1/p} \big\} \\
&\leqslant 4\sum_{j=1}^{k_{n}} \mathbb{V} X_{n,j} I_{\left\{\abs{X_{n,j}} \leqslant \delta \right\}} + 4t^{1/p}\sum_{j=1}^{k_{n}} \mathbb{E} \lvert X_{n,j} \rvert I_{\left\{\lvert X_{n,j} \rvert > \delta \right\}} + 2 t^{2/p} \sum_{j=1}^{k_{n}} \mathbb{P} \big\{\lvert X_{n,j} \rvert > t^{1/p} \big\}
\end{align*}
for all $t \geqslant (\delta/\rho)^{p} \vee \delta^{p}$, lead to
\begin{align*}
&\mathbb{P} \left\{\max_{1 \leqslant i \leqslant k_{n}} \abs{\sum_{j=1}^{i} \left[g_{t^{1/p}}(X_{n,j}) - \mathbb{E} g_{t^{1/p}}(X_{n,j}) \right]} > \frac{C_{\delta} t^{1/p}}{2} \right\} \\
&\quad \leqslant \alpha_{n} \mathbb{P} \left\{\max_{1 \leqslant j \leqslant k_{n}} \lvert g_{t^{1/p}}(X_{n,j}) - \mathbb{E} g_{t^{1/p}}(X_{n,j}) \rvert > 3 \rho t^{1/p} \right\} + \frac{\beta_{n} 2^{2q} q^{q}}{C_{\delta}^{2q} C_{2}^{q}} \left[\frac{s_{n}(t^{1/p})}{t^{2/p}} \right]^{q} \\
&\quad \leqslant \alpha_{n} \mathbb{P} \left\{\max_{1 \leqslant j \leqslant k_{n}} \lvert g_{t^{1/p}}(X_{n,j}) \rvert > \rho t^{1/p} \right\} + \frac{\beta_{n} 2^{3q-1} q^{q}}{C_{\delta}^{2q} C_{2}^{q}} \left(\frac{4}{t^{2/p}} \sum_{j=1}^{k_{n}} \mathbb{V} X_{n,j} I_{\left\{\abs{X_{n,j}} \leqslant \delta \right\}} \right)^{q} \\
&\qquad + \frac{\beta_{n} 2^{3q-1} q^{q}}{C_{\delta}^{2q} C_{2}^{q}} \left(\frac{4}{t^{1/p}} \sum_{j=1}^{k_{n}} \mathbb{E} \lvert X_{n,j} \rvert I_{\left\{\lvert X_{n,j} \rvert > \delta \right\}} + 2 \sum_{j=1}^{k_{n}} \mathbb{P} \big\{\lvert X_{n,j} \rvert > t^{1/p} \big\} \right)^{q} \\
&\quad \leqslant \alpha_{n} \sum_{j=1}^{k_{n}} \mathbb{P} \left\{\lvert g_{t^{1/p}}(X_{n,j}) \rvert > \rho t^{1/p} \right\} + \frac{\beta_{n} 2^{5q-1} q^{q}}{C_{\delta}^{2q} C_{2}^{q}t^{2q/p}} \left( \sum_{j=1}^{k_{n}} \mathbb{V} X_{n,j} I_{\left\{\abs{X_{n,j}} \leqslant \delta \right\}} \right)^{q} \\
&\qquad + \frac{\beta_{n} 2^{4q-2} q^{q}}{C_{\delta}^{2q} C_{2}^{q}} \left(\frac{4}{t^{1/p}} \sum_{j=1}^{k_{n}} \mathbb{E} \lvert X_{n,j} \rvert I_{\left\{\lvert X_{n,j} \rvert > \delta \right\}} \right)^{q} + \frac{\beta_{n} 2^{4q-2} q^{q}}{C_{\delta}^{2q} C_{2}^{q}} \left(2 \sum_{j=1}^{k_{n}} \mathbb{P} \big\{\lvert X_{n,j} \rvert > t^{1/p} \big\} \right)^{q} \\
&\quad \leqslant \alpha_{n} \sum_{j=1}^{k_{n}} \mathbb{P} \left\{\lvert X_{n,j} \rvert > \rho t^{1/p} \right\} + \frac{\beta_{n} 2^{5q-1} q^{q}}{C_{\delta}^{2q} C_{2}^{q}t^{2q/p}} \left( \sum_{j=1}^{k_{n}} \mathbb{V} X_{n,j} I_{\left\{\abs{X_{n,j}} \leqslant \delta \right\}} \right)^{q} \\
& \qquad + \frac{\beta_{n} 2^{6q-2} q^{q}}{C_{\delta}^{2q} C_{2}^{q} t^{q/p}} \left(\sum_{j=1}^{k_{n}} \mathbb{E} \lvert X_{n,j} \rvert I_{\left\{\lvert X_{n,j} \rvert > \delta \right\}} \right)^{q} + \frac{\beta_{n} 2^{5q-2} q^{q}}{C_{\delta}^{2q} C_{2}^{q}} \left(\sum_{j=1}^{k_{n}} \mathbb{P} \big\{\lvert X_{n,j} \rvert > t^{1/p} \big\} \right)^{q}
\end{align*}
for every $t \geqslant (\delta/\rho)^{p} \vee \delta^{p}$ and $n$ large enough. Thereby,
\begin{align*}
&\sum_{n=1}^{\infty} c_{n} \int_{(\delta/\rho)^{p} \vee \delta^{p}}^{\infty} \mathbb{P} \left\{\max_{1 \leqslant i \leqslant k_{n}} \abs{\sum_{j=1}^{i} \left[g_{t^{1/p}}(X_{n,j}) - \mathbb{E} g_{t^{1/p}}(X_{n,j}) \right]} > \frac{C_{\delta}t^{1/p}}{2} \right\}\mathrm{d}t \\
&\quad \leqslant \sum_{n=1}^{\infty} c_{n} \alpha_{n} \int_{(\delta/\rho)^{p}}^{\infty} \sum_{j=1}^{k_{n}} \mathbb{P} \big\{\lvert X_{n,j} \rvert > \rho t^{1/p} \big\} \mathrm{d}t + \frac{2^{5q-1} q^{q}}{C_{\delta}^{2q} C_{2}^{q}} \sum_{n=1}^{\infty} c_{n} \beta_{n} \int_{\delta^{p}}^{\infty} \frac{\mathrm{d}t}{t^{2q/p}} \left(\sum_{j=1}^{k_{n}} \mathbb{V} X_{n,j} I_{\left\{\abs{X_{n,j}} \leqslant \delta \right\}} \right)^{q} \\
& \qquad + \frac{2^{6q-2} q^{q}}{C_{\delta}^{2q} C_{2}^{q}} \sum_{n=1}^{\infty} c_{n} \beta_{n} \int_{\delta^{p}}^{\infty} \frac{\mathrm{d}t}{t^{q/p}} \left(\sum_{j=1}^{k_{n}} \mathbb{E} \lvert X_{n,j} \rvert I_{\left\{\lvert X_{n,j} \rvert > \delta \right\}} \right)^{q} \\
&\qquad + \frac{2^{5q-2} q^{q}}{C_{\delta}^{2q} C_{2}^{q}} \sum_{n=1}^{\infty} c_{n} \beta_{n} \int_{\delta^{p}}^{\infty} \left(\sum_{j=1}^{k_{n}} \mathbb{P} \big\{\lvert X_{n,j} \rvert^{p} > t \big\} \mathrm{d}t \right)^{q} \\
&\quad \leqslant \rho^{-p} \sum_{n=1}^{\infty} c_{n} \alpha_{n} \int_{\delta^{p}}^{\infty} \sum_{j=1}^{k_{n}} \mathbb{P} \big\{\lvert X_{n,j} \rvert^{p} > u \big\} \mathrm{d}u + \frac{p 2^{5q-2}q^{q} \delta^{p - 2q}}{(2q - p) C_{\delta}^{2q} C_{2}^{q}} \sum_{n=1}^{\infty} c_{n} \beta_{n} \left(\sum_{j=1}^{k_{n}} \mathbb{V} X_{n,j} I_{\left\{\abs{X_{n,j}} \leqslant \delta \right\}} \right)^{q} \\
& \qquad + \frac{p 2^{6q-2}q^{q} \delta^{p - q}}{(q - p) C_{\delta}^{2q} C_{2}^{q}} \sum_{n=1}^{\infty} c_{n} \beta_{n} \left(\sum_{j=1}^{k_{n}} \mathbb{E} \lvert X_{n,j} \rvert I_{\left\{\lvert X_{n,j} \rvert > \delta \right\}} \right)^{q} \\
&\qquad + \frac{2^{5q-2} q^{q}}{C_{\delta}^{2q} C_{2}^{q}}\sum_{n=1}^{\infty} c_{n} \beta_{n} \int_{\delta^{p}}^{\infty} \left(\sum_{j=1}^{k_{n}} \mathbb{P} \big\{\lvert X_{n,j} \rvert^{p} > t \big\} \mathrm{d}t \right)^{q} \\
&\quad < \infty,
\end{align*}
and the proof is concluded.
\end{proof}

By means of Theorems~\ref{thr:2} and~\ref{thr:3}, we state the following corollary.

\begin{corollary}\label{cor:1}
Let $p \geqslant 1$, $\{c_{n} \}$ be a sequence of positive numbers, and $\left\{X_{n,j}, \, 1 \leqslant j \leqslant k_{n}, n \geqslant 1 \right\}$ be an array of random variables satisfying \eqref{eq:2.1} for sequences $\{\alpha_{n} \}$, $\{\beta_{n} \}$ of nonnegative numbers such that $\alpha_{n} = O(1) = \beta_{n}$ as $n \rightarrow \infty$. If assumption \textnormal{(i')} holds, and there exists a constant $\delta > 0$ such that \\

\noindent \textnormal{(b')} $\sum_{j=1}^{k_{n}} \mathbb{P} \left\{\lvert X_{n,j} \rvert > \delta \right\} = o(1)$ as $n \rightarrow \infty$, \\

\noindent \textnormal{(c')} $\sum_{n=1}^{\infty} \sum_{j=1}^{k_{n}} c_{n} \int_{\delta^{p}}^{\infty} \mathbb{P} \left\{\lvert X_{n,j} \rvert^{p} > t \right\} \mathrm{d}t < \infty$, \\

\noindent \textnormal{(d')} for some $q > p$,
\begin{equation*}
\sum_{n=1}^{\infty} c_{n} \left(\sum_{j=1}^{k_{n}} \mathbb{V} X_{n,j} I_{\left\{\lvert X_{n,j} \rvert \leqslant \delta \right\}} \right)^{q} < \infty \quad and \quad \sum_{n=1}^{\infty} c_{n} \left(\sum_{j=1}^{k_{n}} \mathbb{E} \lvert X_{n,j} \rvert I_{\left\{\lvert X_{n,j} \rvert > \delta \right\}} \right)^{q} < \infty,
\end{equation*}

\noindent then, for all $\varepsilon > 0$, \eqref{eq:2.10} holds.
\end{corollary}

\begin{proof}
Since (b') ensures $\sum_{j=1}^{k_{n}} \mathbb{P} \left\{\lvert X_{n,j} \rvert^{p} > t \right\} \leqslant \sum_{j=1}^{k_{n}} \mathbb{P} \left\{\lvert X_{n,j} \rvert > \delta \right\} \leqslant 1$ for any $t \geqslant \delta^{p}$ and $n$ large enough, assumption (c') entails
\begin{equation*}
\sum_{n=1}^{\infty} c_{n} \beta_{n} \int_{\delta^{p}}^{\infty} \left(\sum_{j=1}^{k_{n}} \mathbb{P} \big\{\lvert X_{n,j} \rvert^{p} > t \big\} \mathrm{d}t \right)^{q} \leqslant C \sum_{n=1}^{\infty} c_{n} \int_{\delta^{p}}^{\infty} \sum_{j=1}^{k_{n}} \mathbb{P} \big\{\lvert X_{n,j} \rvert^{p} > t \big\} \mathrm{d}t < \infty
\end{equation*}
where $C$ is some positive constant (non-depending on $n$) such that $\beta_{n} \leqslant C$ for all $n$. The thesis is a consequence of Theorems~\ref{thr:2} and~\ref{thr:3} by noting that (b') implies (b).
\end{proof}

It is worthy to note that assumption (i) of Theorem 3.1 in \cite{Shen16} implies both conditions (i') and (c') of Corollary~\ref{cor:1} when $\alpha_{n} = 2$ and $\beta_{n} = 8$. Indeed, the example below shows that $\int_{\delta^{p}}^{\infty} \mathbb{P} \big\{\lvert X_{n,j} \rvert^{p} > t \big\} \, \mathrm{d}t$ can be of smaller order than $\mathbb{E} \lvert X_{n,j} \rvert^{p} I_{\left\{\lvert X_{n,j} \rvert > \delta \right\}}$ on the one hand, and on the other hand, $\mathbb{P} \big\{\lvert X_{n,j} \rvert > \varepsilon \big\}$ can also be of smaller order than $\mathbb{E}\lvert X_{n,j} \rvert^{p} I_{\left\{\lvert X_{n,j} \rvert > \varepsilon \right\}}$.

\begin{example}
Let $\{\xi_{n}, \, n \geqslant 1 \}$ be a sequence of i.i.d. random variables and $X_{n,j} := \xi_{j}/b_{n}$ where $\{b_{n} \}$ is a sequence of positive constants satisfying $b_{n} \rightarrow \infty$ as $n \rightarrow \infty$. If the tail distribution of $\lvert \xi_{1} \rvert^{p}$ is rapidly varying with index $-\infty$ (e.g. $\lvert \xi_{1} \rvert^{p}$ having exponential distribution) and $\delta > 0$, then
\begin{equation*}
\frac{\int_{\delta^{p}}^{\infty} \mathbb{P} \big\{\lvert X_{n,j} \rvert^{p} > t \big\} \, \mathrm{d}t}{\mathbb{E} \lvert X_{n,j} \rvert^{p} I_{\left\{\lvert X_{n,j} \rvert > \delta \right\}}} \leqslant \frac{\int_{(\delta b_{n})^{p}}^{\infty} \mathbb{P} \{\lvert \xi_{1} \rvert^{p} > u \} \, \mathrm{d}u}{(\delta b_{n})^{p} \mathbb{P} \{\lvert \xi_{1} \rvert > \delta b_{n} \}} = o(1), \quad n \rightarrow \infty
\end{equation*}
(see \cite{Embrechts97}, page $570$), i.e. $\int_{\delta^{p}}^{\infty} \mathbb{P} \big\{\lvert X_{n,j} \rvert^{p} > t \big\} \, \mathrm{d}t = o \left(\mathbb{E} \lvert X_{n,j} \rvert^{p} I_{\left\{\lvert X_{n,j} \rvert > \delta \right\}} \right)$, $n \rightarrow \infty$. Instead, if $\mathbb{P} \{\lvert \xi_{1} \rvert^{p} > t \} \sim L(t)/t$ as $t \rightarrow \infty$, where $L$ is a slowly varying function at $\infty$ locally bounded in $[t_{0},\infty)$, for some $t_{0} \geqslant 0$, and such that $\int_{t_{0}}^{\infty} L(t)/t \, \mathrm{d}t < \infty$, then, for each fixed $\varepsilon > 0$,
\begin{equation*}
\frac{\int_{\varepsilon^{p}}^{\infty} \mathbb{P} \{\lvert X_{n,j} \rvert^{p} > t \} \, \mathrm{d}t}{\varepsilon^{p} \mathbb{P} \{\lvert X_{n,j} \rvert > \varepsilon \}} = \frac{\int_{(\varepsilon b_{n})^{p}}^{\infty} \mathbb{P} \{\lvert \xi_{1} \rvert^{p} > u \} \, \mathrm{d}u}{(\varepsilon b_{n})^{p} \mathbb{P} \{\lvert \xi_{1} \rvert > \varepsilon b_{n} \}} \longrightarrow \infty, \quad n \rightarrow \infty
\end{equation*}
(see \cite{Embrechts97}, page $567$) and whence $\mathbb{P} \big\{\lvert X_{n,j} \rvert > \varepsilon \big\} = o\left(\mathbb{E}\lvert X_{n,j} \rvert^{p} I_{\left\{\lvert X_{n,j} \rvert > \varepsilon \right\}} \right)$, $n \rightarrow \infty$.
\end{example}

\section{Change-point estimator}\label{sec:3}

\indent

In the nineties, the problem of a change in the mean of a sequence of observations received attention by many authors (see \cite{Antoch95}, \cite{Antoch96}, \cite{Csorgo97}, \cite{Kokoszka98} among others). Let $\{\mu_{n,k}, \, 1 \leqslant k \leqslant n, n \geqslant 1 \}$ be a triangular array of real numbers and $\{X_{n,k}, 1 \leqslant k \leqslant n, \, n \geqslant 1  \}$ be a triangular array of zero-mean random variables. Consider the model
\begin{equation}\label{eq:3.1}
Y_{n,k} =
\begin{cases}
  \mu_{n,k} + X_{n,k}, & 1 \leqslant k \leqslant k^{\ast} \\[5pt]
  \mu_{n,k} + \Delta_{n} + X_{n,k}, & k^{\ast} < k \leqslant n
\end{cases}
\end{equation}
where $k^{\ast} = [n \tau^{\ast}]$ is an unknown change-point, $\Delta_{n}$ is the (unknown) change-amount, and $\tau^{\ast} \in (0,1)$ is fixed (here, $[ x ]$ stands for the integer part of $x$). In \cite{Kokoszka98}, the authors proved the (weak) consistency of CUSUM-type estimator of the point of shift in the mean of a sequence of observations, i.e. by considering the estimators
\begin{equation*}
\widehat{k}_{n} := \min\left\{k\colon \lvert U_{n,k} \rvert = \max_{1 \leqslant j < n} \lvert U_{n,j} \rvert \right\}
\end{equation*}
and $\widehat{\tau}_{n} := \widehat{k}_{n}/n$ of $k^{\ast}$ and $\tau^{\ast}$, respectively, with
\begin{equation*}
U_{n,k} := \frac{(n - k)^{1 - \gamma}}{n^{1 - \gamma} k^{\gamma}} \sum_{j=1}^{k} Y_{n,j} - \frac{k^{1 - \gamma}}{n^{1 - \gamma} (n - k)^{\gamma}} \sum_{j=k+1}^{n} Y_{n,j} \qquad (0 \leqslant \gamma < 1)
\end{equation*}
conditions were given for $\widehat{\tau}_{n}$ to be weakly consistent of $\tau^{\ast}$. The next statement gives us the complete convergence, in the sense of Hsu and Robbins \cite{Hsu47}, of $\widehat{\tau}_{n}$, i.e. we provide conditions under which $\widehat{\tau}_{n}$ is \emph{completely consistent} for $\tau^{\ast}$. We assume that $C(\, \cdot \,)$ denotes a positive generic constant depending (only) on the arguments specified within the parentheses that may assume different values at each appearance.

\begin{theorem}\label{thr:4}
Suppose that in model \eqref{eq:3.1}, $\Delta_{n} \neq 0$ for any $n \geqslant 1$, and $\left\{X_{n,j}, \, 1 \leqslant j \leqslant n, n \geqslant 1 \right\}$ is an array of zero-mean random variables satisfying \eqref{eq:2.1} for sequences $\{\alpha_{n} \}$, $\{\beta_{n} \}$ of nonnegative numbers such that $\alpha_{n} = O(1) = \beta_{n}$ as $n \rightarrow \infty$. If either \textnormal{(a)} $\sup_{n,j} \mathbb{E} \lvert X_{n,j} \rvert^{r} < \infty$ for $1 < r \leqslant 2$ verifying
\begin{equation}\label{eq:3.2}
\begin{cases}
  \sum_{n=1}^{\infty} \lvert \Delta_{n} \rvert^{-r} n^{1 - r}, & 0 \leqslant \gamma < 1/r \\[5pt]
  \sum_{n=1}^{\infty} \lvert \Delta_{n} \rvert^{-r} n^{1 - r} \log n, & \gamma = 1/r \\[5pt]
  \sum_{n=1}^{\infty} \lvert \Delta_{n} \rvert^{-r} n^{r(\gamma - 1)}, & 1/r < \gamma < 1,
\end{cases}
\end{equation}
or \textnormal{(b)} $\sup_{n,j} \mathbb{E} \lvert X_{n,j} \rvert^{r} < \infty$ for some $r > 2$ satisfying
\begin{equation}\label{eq:3.3}
\begin{cases}
  \sum_{n=1}^{\infty} \left[\lvert \Delta_{n} \rvert^{-r} + \lvert \Delta_{n} \rvert^{2(1 - r)} \right] n^{1 - r} < \infty, & 0 \leqslant \gamma < 1/r \\[5pt]
  \sum_{n=1}^{\infty} \left[\lvert \Delta_{n} \rvert^{-r} \log n + \lvert \Delta_{n} \rvert^{2(1 - r)} \right] n^{1 - r} < \infty, & \gamma = 1/r \\[5pt]
  \sum_{n=1}^{\infty} \left[\lvert \Delta_{n} \rvert^{-r} + \lvert \Delta_{n} \rvert^{2r(\gamma - 1)} \right] n^{r(\gamma - 1)} < \infty, & 1/r < \gamma < 1/2 \\[5pt]
  \sum_{n=1}^{\infty} \left[\lvert \Delta_{n} \rvert^{-r} n^{r/2 - 1} + \lvert \Delta_{n} \rvert^{2(1 - r)} \log^{r - 1} n \right] n^{1 - r} < \infty, & \gamma = 1/2 \\[5pt]
  \sum_{n=1}^{\infty} \lvert \Delta_{n} \rvert^{-r} n^{r(\gamma - 1)} < \infty, & 1/2 < \gamma < 1,
\end{cases}
\end{equation}
then $\widehat{\tau}_{n}$ converges completely to $\tau^{\ast}$. Furthermore, for all $\varepsilon > 0$, $\sum_{n=1}^{\infty} \mathbb{E}\left(\lvert \widehat{\tau}_{n} - \tau^{\ast} \rvert - \varepsilon \right)_{+}^{r} < \infty$.
\end{theorem}

\begin{proof}
As in proof of Theorem 1.1 of \cite{Kokoszka98} (see \cite{Kokoszka98}, page $391$), it can be shown that
\begin{equation*}
\lvert \Delta_{n} \rvert (1 - \gamma)(\tau^{\ast})^{-\gamma}(1 - \tau^{\ast}) \min\{\tau^{\ast},1 - \tau^{\ast} \} \lvert \tau^{\ast} - \widehat{\tau}_{n} \rvert \leqslant 2 n^{\gamma - 1} \max_{1 \leqslant k < n} \lvert U_{n,k} - \mathbb{E} U_{n,k} \rvert,
\end{equation*}
whence
\begin{equation}\label{eq:3.4}
\lvert \tau^{\ast} - \widehat{\tau}_{n} \rvert \leqslant C(\gamma, \tau^{\ast}) \frac{n^{\gamma - 1}}{\lvert \Delta_{n} \rvert} \max_{1 \leqslant k < n} \lvert U_{n,k} - \mathbb{E} U_{n,k} \rvert.
\end{equation}
By employing Abel's lemma (see \cite{Chow97}, page $114$), we have
\begin{equation}\label{eq:3.5}
\begin{split}
&\max_{1 \leqslant k < n} \lvert U_{n,k} - \mathbb{E} U_{n,k} \rvert \\
&\quad \leqslant \max_{1 \leqslant k < n} \frac{1}{k^{\gamma}} \abs{\sum_{j=1}^{k} (Y_{n,j} - \mathbb{E} Y_{n,j})} + \max_{1 \leqslant k < n} \frac{1}{(n - k)^{\gamma}} \abs{\sum_{j=k+1}^{n} (Y_{n,j} - \mathbb{E} Y_{n,j})} \\
&\quad = \max_{1 \leqslant k < n} \frac{1}{k^{\gamma}} \abs{\sum_{j=1}^{k} (Y_{n,j} - \mathbb{E} Y_{n,j})} + \max_{1 \leqslant k < n} \frac{1}{k^{\gamma}} \abs{\sum_{j=1}^{k} (Y_{n,n - j + 1} - \mathbb{E} Y_{n,n - j + 1})} \\
&\quad = \max_{1 \leqslant k < n} \frac{1}{k^{\gamma}} \abs{\sum_{j=1}^{k} X_{n,j}} + \max_{1 \leqslant k < n} \frac{1}{k^{\gamma}} \abs{\sum_{j=1}^{k} X_{n,n - j + 1}} \\
&\quad \leqslant 2 \max_{1 \leqslant k < n} \abs{\sum_{j=1}^{k} \frac{X_{n,j}}{j^{\gamma}}} + 2 \max_{1 \leqslant k < n} \abs{\sum_{j=1}^{k} \frac{X_{n,n - j + 1}}{j^{\gamma}}}
\end{split}
\end{equation}
and it suffices to prove that, for any $\varepsilon > 0$,
\begin{gather}
\sum_{n=2}^{\infty} \mathbb{P} \left\{ \max_{1 \leqslant k < n} \abs{\sum_{j=1}^{k} \frac{n^{\gamma - 1} X_{n,j}}{\lvert \Delta_{n} \rvert j^{\gamma}}} > \varepsilon \right\} < \infty, \label{eq:3.6} \\
\sum_{n=2}^{\infty} \mathbb{P} \left\{ \max_{1 \leqslant k < n} \abs{\sum_{j=1}^{k} \frac{n^{\gamma - 1} X_{n,n - j + 1}}{\lvert \Delta_{n} \rvert j^{\gamma}}} > \varepsilon \right\} < \infty. \label{eq:3.7}
\end{gather}
For \eqref{eq:3.6} notice that $\mathbb{E} X_{n,k} = 0$ and for each $s > 0$,
\begin{equation}\label{eq:3.8}
\sum_{j=1}^{n} j^{-s} \leqslant
\begin{cases}
  C(s) n^{1 - s}, & s < 1 \\[5pt]
  C(s) \log n, & s = 1 \\[5pt]
  C(s), & s > 1.
\end{cases}
\end{equation}
Thus,
\begin{align*}
\max_{1 \leqslant k < n} \abs{\sum_{j=1}^{k} \mathbb{E} \left(\frac{n^{\gamma - 1} X_{n,j}}{\lvert \Delta_{n} \rvert j^{\gamma}} I_{\left\{\frac{n^{\gamma - 1} \lvert X_{n,j} \rvert}{\lvert \Delta_{n} \rvert j^{\gamma}} \leqslant \delta \right\}} \right)} &= \max_{1 \leqslant k < n} \abs{\sum_{j=1}^{k} \mathbb{E} \left(\frac{n^{\gamma - 1} X_{n,j}}{\lvert \Delta_{n} \rvert j^{\gamma}} I_{\left\{\frac{n^{\gamma - 1} \lvert X_{n,j} \rvert}{\lvert \Delta_{n} \rvert j^{\gamma}} > \delta \right\}} \right)} \\
&\leqslant \sum_{j=1}^{n-1} \frac{n^{\gamma - 1}}{\lvert \Delta_{n} \rvert j^{\gamma}} \mathbb{E} \lvert X_{n,j} \rvert I_{\left\{\frac{n^{\gamma - 1} \lvert X_{n,j} \rvert}{\lvert \Delta_{n} \rvert j^{\gamma}} > \delta \right\}} \\
&\leqslant \delta^{1 - r} \sum_{j=1}^{n-1} \frac{n^{r(\gamma - 1)}}{\lvert \Delta_{n} \rvert^{r} j^{r\gamma}} \mathbb{E} \lvert X_{n,j} \rvert^{r} \\
&\leqslant \delta^{1 - r} \frac{n^{r(\gamma - 1)}}{\lvert \Delta_{n} \rvert^{r}} \bigg(\sup_{n,j} \mathbb{E} \lvert X_{n,j} \rvert^{r} \bigg) \sum_{j=1}^{n-1} j^{-r\gamma} \\
&\leqslant
\begin{cases}
  C(r,\delta,\gamma) \lvert \Delta_{n} \rvert^{-r} n^{1 - r}, & 0 \leqslant \gamma < 1/r \\[5pt]
  C(r,\delta) \lvert \Delta_{n} \rvert^{-r} n^{1 - r} \log n, & \gamma = 1/r \\[5pt]
  C(r,\delta,\gamma) \lvert \Delta_{n} \rvert^{-r} n^{r(\gamma - 1)}, & 1/r < \gamma < 1
\end{cases}
\end{align*}
and \eqref{eq:3.2}-\eqref{eq:3.3} imply
\begin{equation}\label{eq:3.9}
\max_{1 \leqslant k < n} \abs{\sum_{j=1}^{k} \mathbb{E} \left(\frac{n^{\gamma - 1} X_{n,j}}{\lvert \Delta_{n} \rvert j^{\gamma}} I_{\left\{\frac{n^{\gamma - 1} \lvert X_{n,j} \rvert}{\lvert \Delta_{n} \rvert j^{\gamma}} \leqslant \delta \right\}} \right)} = o(1), \quad n \rightarrow \infty.
\end{equation}
According to Markov inequality and \eqref{eq:3.8}, we get
\begin{equation}\label{eq:3.10}
\begin{split}
\sum_{j=1}^{n-1} \mathbb{P} \{\lvert X_{n,j} \rvert > \varepsilon n^{1 - \gamma} j^{\gamma} \lvert \Delta_{n} \rvert \} &\leqslant \sum_{j=1}^{n-1} \frac{n^{r(\gamma - 1)} \mathbb{E} \lvert X_{n,j} \rvert^{r}}{\varepsilon^{r} \lvert \Delta_{n} \rvert^{r} j^{r \gamma}} \\
&\leqslant \frac{n^{r(\gamma - 1)}}{\varepsilon^{r} \lvert \Delta_{n} \rvert^{r}} \sup_{n,j} \mathbb{E} \lvert X_{n,j} \rvert^{r} \sum_{j=1}^{n-1} j^{-r \gamma} \\
&\leqslant
\begin{cases}
  C(r,\varepsilon,\gamma) \lvert \Delta_{n} \rvert^{-r} n^{1 - r}, & 0 \leqslant \gamma < 1/r \\[5pt]
  C(r,\varepsilon) \lvert \Delta_{n} \rvert^{-r} n^{1 - r} \log n, & \gamma = 1/r \\[5pt]
  C(r,\varepsilon,\gamma) \lvert \Delta_{n} \rvert^{-r} n^{r(\gamma - 1)}, & 1/r < \gamma < 1
\end{cases}
\end{split}
\end{equation}
and \eqref{eq:3.2}-\eqref{eq:3.3} ensure assumption (i') in Theorem~\ref{thr:2}. Moreover,
\begin{align*}
\left[\sum_{j=1}^{n-1} \mathbb{V} \left(\frac{n^{\gamma - 1} X_{n,j}}{\lvert \Delta_{n} \rvert j^{\gamma}}I_{ \left\{\frac{n^{\gamma - 1} \lvert X_{n,j} \rvert}{\lvert \Delta_{n} \rvert j^{\gamma}} \leqslant \delta \right\}} \right) \right]^{q} &\leqslant \left[\sum_{j=1}^{n-1} \frac{n^{2(\gamma - 1)}}{\Delta_{n}^{2} j^{2\gamma}} \mathbb{E} X_{n,j}^{2} I_{ \left\{\frac{n^{\gamma - 1} \lvert X_{n,j} \rvert}{\lvert \Delta_{n} \rvert j^{\gamma}} \leqslant \delta \right\}} \right]^{q} \\
&\leqslant \delta^{q(2 - r)} \left[\sum_{j=1}^{n-1} \frac{n^{r(\gamma - 1)}}{\lvert \Delta_{n} \rvert^{r} j^{r\gamma}} \mathbb{E} \lvert X_{n,j} \rvert^{r} \right]^{q} \\
&\leqslant C(r,\delta,q) \frac{n^{q r(\gamma - 1)}}{\lvert \Delta_{n} \rvert^{qr}} \bigg(\sup_{n,j} \mathbb{E} \lvert X_{n,j} \rvert^{r} \bigg)^{q} \left(\sum_{j=1}^{n-1} j^{-r\gamma} \right)^{q} \\
&\leqslant
\begin{cases}
  C(r,q,\delta,\gamma) \lvert \Delta_{n} \rvert^{-qr} n^{q(1 - r)}, & 0 \leqslant \gamma < 1/r \\[5pt]
  C(r,q,\delta) \lvert \Delta_{n} \rvert^{-qr} n^{q(1 - r)} \log^{q} n, & \gamma = 1/r \\[5pt]
  C(r,q,\delta,\gamma) \lvert \Delta_{n} \rvert^{-qr} n^{qr(\gamma - 1)}, & 1/r < \gamma < 1
\end{cases}
\end{align*}
when $r \leqslant 2$, and
\begin{align*}
\left[\sum_{j=1}^{n-1} \mathbb{V} \left(\frac{n^{\gamma - 1} X_{n,j}}{\lvert \Delta_{n} \rvert j^{\gamma}}I_{ \left\{\frac{n^{\gamma - 1} \lvert X_{n,j} \rvert}{\lvert \Delta_{n} \rvert j^{\gamma}} \leqslant \delta \right\}} \right) \right]^{q} &\leqslant \left[\sum_{j=1}^{n-1} \frac{n^{2(\gamma - 1)}}{\Delta_{n}^{2} \, j^{2\gamma}} \mathbb{E} X_{n,j}^{2} I_{ \left\{\frac{n^{\gamma - 1} \lvert X_{n,j} \rvert}{\lvert \Delta_{n} \rvert j^{\gamma}} \leqslant \delta \right\}} \right]^{q} \\
&\leqslant \left(\sup_{n,j} \mathbb{E} \lvert X_{n,j} \rvert^{r} \right)^{2q/r} \frac{n^{2q(\gamma - 1)}}{\lvert \Delta_{n} \rvert^{2q}} \left(\sum_{j=1}^{n-1} j^{-2\gamma} \right)^{q} \\
&\leqslant
\begin{cases}
  C(q,r,\gamma) \lvert \Delta_{n} \rvert^{-2q} n^{-q}, & 0 \leqslant \gamma < 1/2 \\[5pt]
  C(q,r) \lvert \Delta_{n} \rvert^{-2q} n^{-q} \log^{q} n, & \gamma = 1/2 \\[5pt]
  C(q,r,\gamma) \lvert \Delta_{n} \rvert^{-2q} n^{2q(\gamma - 1)}, & 1/2 < \gamma < 1
\end{cases}
\end{align*}
whenever $r > 2$. By taking
\begin{equation*}
q =
\begin{cases}
  r(1 - \gamma), & 1/r < \gamma < 1/2 \\
  r/2, & 1/2 < \gamma < 1 \\
  r - 1, & \text{otherwise} \\
\end{cases}
\end{equation*}
in the latter inequality and any $q \geqslant 1$ in the former inequality, it follows
\begin{equation}\label{eq:3.11}
\sum_{n=2}^{\infty} \left[\sum_{j=1}^{n-1} \mathbb{V} \left(\frac{n^{\gamma - 1} X_{n,j}}{\lvert \Delta_{n} \rvert j^{\gamma}}I_{ \left\{\frac{n^{\gamma - 1}  \lvert X_{n,j} \rvert}{\lvert \Delta_{n} \rvert j^{\gamma}} \leqslant \delta \right\}} \right) \right]^{q} < \infty
\end{equation}
via \eqref{eq:3.2}-\eqref{eq:3.3}. Hence, \eqref{eq:3.6} holds from Theorem~\ref{thr:2} with $c_{n} = 1$ and \eqref{eq:3.9}. Similarly, it can be shown \eqref{eq:3.7}; the details are omitted. \\

\noindent For all reals numbers $x,y$ and $p,\varepsilon > 0$, we have
\begin{equation*}
(\lvert x + y \rvert - \varepsilon)_{+}^{p} \leqslant \max\{1,2^{p - 1} \} \left[\left(\lvert x \rvert - \frac{\varepsilon}{2} \right)_{+}^{p} + \left(\lvert y \rvert - \frac{\varepsilon}{2} \right)_{+}^{p} \right],
\end{equation*}
whence, from \eqref{eq:3.4} and \eqref{eq:3.5}, it remains to prove
\begin{gather}
\sum_{n=2}^{\infty} \mathbb{E} \left[\max_{1 \leqslant k < n} \abs{\sum_{j=1}^{k} \frac{n^{\gamma - 1} X_{n,j}}{\lvert \Delta_{n} \rvert j^{\gamma}}} - \varepsilon \right]_{+}^{r} < \infty, \label{eq:3.12} \\
\sum_{n=2}^{\infty} \mathbb{P} \left[\max_{1 \leqslant k < n} \abs{\sum_{j=1}^{k} \frac{n^{\gamma - 1} X_{n,n - j + 1}}{\lvert \Delta_{n} \rvert j^{\gamma}}} - \varepsilon \right]_{+}^{r} < \infty \label{eq:3.13}
\end{gather}
for every $\varepsilon > 0$. In addition to \eqref{eq:3.10}, we still have
\begin{align*}
\sum_{j=1}^{n-1} \mathbb{P} \left\{\frac{n^{1 - \gamma} \lvert X_{n,j} \rvert}{\lvert \Delta_{n} \rvert j^{\gamma}} > \delta \right\} &\leqslant
\begin{cases}
  C(r,\delta,\gamma) \lvert \Delta_{n} \rvert^{-r} n^{1 - r}, & 0 \leqslant \gamma < 1/r \\[5pt]
  C(r,\delta) \lvert \Delta_{n} \rvert^{-r} n^{1 - r} \log n, & \gamma = 1/r \\[5pt]
  C(r,\delta,\gamma) \lvert \Delta_{n} \rvert^{-r} n^{r(\gamma - 1)}, & 1/r < \gamma < 1,
\end{cases}
\end{align*}
and
\begin{align*}
\sum_{j=1}^{n-1} \int_{\delta^{r}}^{\infty} \mathbb{P} \left\{\frac{n^{r(\gamma - 1)} \lvert X_{n,j} \rvert^{r}}{\lvert \Delta_{n} \rvert^{r} j^{r\gamma}} > t \right\} \mathrm{d}t &\leqslant \sum_{j=1}^{n-1} \frac{n^{r(\gamma - 1)} \mathbb{E} \lvert X_{n,j} \rvert^{r}}{\lvert \Delta_{n} \rvert^{r} j^{r\gamma}} I_{\left\{\frac{n^{\gamma - 1} \lvert X_{n,j} \rvert}{\lvert \Delta_{n} \rvert j^{\gamma}} > \delta \right\}} \\
&\leqslant \frac{n^{r(\gamma - 1)}}{\lvert \Delta_{n} \rvert^{r}} \bigg(\sup_{n,j} \mathbb{E} \lvert X_{n,j} \rvert^{r} \bigg) \sum_{j=1}^{n-1} j^{-r \gamma} \\
&\leqslant
\begin{cases}
  C(r,\gamma) \lvert \Delta_{n} \rvert^{-r} n^{1 - r}, & 0 \leqslant \gamma < 1/r \\[5pt]
  C(r) \lvert \Delta_{n} \rvert^{-r} n^{1 - r} \log n, & \gamma = 1/r \\[5pt]
  C(r,\gamma) \lvert \Delta_{n} \rvert^{-r} n^{r(\gamma - 1)}, & 1/r < \gamma < 1.
\end{cases}
\end{align*}
From \eqref{eq:3.2}-\eqref{eq:3.3}, assumptions (b') and (c') in Corollary~\ref{cor:1} hold, as well as (i') in Theorem~\ref{thr:2}. Since \eqref{eq:3.11} remains valid, and
\begin{align*}
\sum_{j=1}^{n-1} \mathbb{E} \left(\frac{n^{\gamma - 1} \lvert X_{n,j} \rvert}{\lvert \Delta_{n} \rvert j^{\gamma}} I_{\left\{\frac{n^{\gamma - 1}\lvert X_{n,j} \rvert}{\lvert \Delta_{n} \rvert j^{\gamma}} > \delta \right\}} \right) &\leqslant \delta^{1 - r} \sum_{j=1}^{n-1} \frac{n^{r(\gamma - 1)} \mathbb{E} \lvert X_{n,j} \rvert^{r}}{\lvert \Delta_{n} \rvert^{r} j^{r\gamma}} \\
&\leqslant \frac{n^{r(\gamma - 1)}}{\delta^{r - 1} \lvert \Delta_{n} \rvert^{r}} \bigg(\sup_{n,j} \mathbb{E} \lvert X_{n,j} \rvert^{r} \bigg) \sum_{j=1}^{n-1} j^{-r \gamma} \\
&\leqslant
\begin{cases}
  C(r,\delta,\gamma) \lvert \Delta_{n} \rvert^{-r} n^{1 - r}, & 0 \leqslant \gamma < 1/r \\[5pt]
  C(r,\delta) \lvert \Delta_{n} \rvert^{-r} n^{1 - r} \log n, & \gamma = 1/r \\[5pt]
  C(r,\delta,\gamma) \lvert \Delta_{n} \rvert^{-r} n^{r(\gamma - 1)}, & 1/r < \gamma < 1,
\end{cases}
\end{align*}
assumption (d') in Corollary~\ref{cor:1} is also met. Hence, \eqref{eq:3.12} follows via Corollary~\ref{cor:1}. Analogously, we can demonstrate \eqref{eq:3.13}. The proof is complete.
\end{proof}

\begin{remark}
If $\Delta_{n} = O(1)$ as $n \rightarrow \infty$, assumption \eqref{eq:3.3} can be replaced by
\begin{equation*}
\begin{cases}
  \sum_{n=1}^{\infty} \lvert \Delta_{n} \rvert^{2(1 - r)} n^{1 - r} < \infty, & 0 \leqslant \gamma < 1/r \\[5pt]
  \sum_{n=1}^{\infty} \lvert \Delta_{n} \rvert^{2(1 - r)} n^{1 - r} \log n < \infty, & \gamma = 1/r \\[5pt]
  \sum_{n=1}^{\infty} \lvert \Delta_{n} \rvert^{2r(\gamma - 1)} n^{r(\gamma - 1)} < \infty, & 1/r < \gamma < 1/2 \\[5pt]
  \sum_{n=1}^{\infty} \lvert \Delta_{n} \rvert^{2(1 - r)} n^{-r/2} < \infty, & \gamma = 1/2 \\[5pt]
  \sum_{n=1}^{\infty} \lvert \Delta_{n} \rvert^{-r} n^{r(\gamma - 1)} < \infty, & 1/2 < \gamma < 1.
\end{cases}
\end{equation*}
Since, for every $\varepsilon > 0$,
\begin{equation*}
\mathbb{E} \lvert \widehat{\tau}_{n} - \tau^{\ast} \rvert^{r} \leqslant 2^{r - 1} \left[\mathbb{E}\left(\lvert \widehat{\tau}_{n} - \tau^{\ast} \rvert - \varepsilon \right)_{+}^{r} + \varepsilon^{r} \right]
\end{equation*}
assumptions of Theorem~\ref{thr:4} also entail that $\widehat{\tau}_{n}$ is consistent in the $r$th mean for $\tau^{\ast}$.
\end{remark}

\noindent In \cite{Shi09}, the authors investigated strong consistency of $\widehat{\tau}_{n}$ by considering the model
\begin{equation}\label{eq:3.14}
Y_{i} =
\begin{cases}
   X_{i}, & i \leqslant k^{\ast} \\[5pt]
   \delta_{0} + X_{i}, & i > k^{\ast}
\end{cases}
\end{equation}
where $\{X_{i}, \, i \geqslant 1 \}$ is either a sequence of independent random variables or a sequence having certain dependence structure. Let us emphasize that, not only model \eqref{eq:3.1} is more general than \eqref{eq:3.14}, but also asymptotic results of \cite{Shi09} are no longer guaranteed in our scenario. As a matter of fact, strong convergence for arrays of random variables requires stronger conditions on (absolute) moments of the random variables (see, for instance, \cite{Moricz89}).

\subsection{Simulations}\label{sec:3.1}

In model (3.1), we shall assume the existence of a change-point $k^{\ast}$ and
\begin{equation}\label{eq:3.15}
(X_{n,1},X_{n,2},\ldots,X_{n,n}) \overset{\mathrm{d}}{=} \mathscr{N}(\boldsymbol{0},\boldsymbol{\Sigma}_{n})
\end{equation}
i.e., the $n$-dimensional random variable $(X_{n,1},X_{n,2},\ldots,X_{n,n})$ has (nonsingular) multivariate normal distribution with mean vector zero and covariance matrix
\begin{equation*}
\boldsymbol{\Sigma}_{n} = \left[
\begin{array}{ccccc}
  1/n + \sigma^{-n - 1}/(\sigma - 1) & -\sigma^{-n - 3} & \ldots & -\sigma^{-2n - 1} \\
  -\sigma^{-n - 3} & 2/n + 2\sigma^{-n - 2}/(\sigma - 1) & \ldots & -\sigma^{-2n - 2} \\
  \vdots & \vdots & \ddots & \vdots \\
  -\sigma^{-2n - 1} & -\sigma^{-2n - 2} & \ldots & 1 + n\sigma^{-2n}/(\sigma - 1)
\end{array}
\right] \quad (\sigma > 1).
\end{equation*}
It is well-known that if \eqref{eq:3.15} holds with $\boldsymbol{\Sigma}_{n}$ having non-positive off-diagonal entries then, for each $n \geqslant 1$, the sequence of random variables $\{X_{n,j}, \, 1 \leqslant j \leqslant n\}$ is negatively associated (see, for instance, Theorem 8 of \cite{Karlin83}). Further, $X_{n,j} \overset{\mathrm{d}}{=} \mathscr{N}\big(0,j/n + j\sigma^{-n - j}/(\sigma - 1) \big)$ (see \cite{Tong90}) and
\begin{equation*}
\mathbb{E} \lvert X_{n,j} \rvert^{r} = \frac{2^{r/2}}{\sqrt{\pi}}  \Gamma\left(\frac{r + 1}{2} \right) \left(\frac{j}{n} + \frac{j \sigma^{-n - j}}{\sigma - 1} \right)^{r/2} \qquad (r > 0)
\end{equation*}
with $\Gamma(z) := \int_{0}^{\infty} t^{z - 1} \mathrm{e}^{-t} \, \mathrm{d}t$, $\mathrm{Re}(z) > 0$ denoting the gamma function.

For the purpose of the simulations, also note that, by taking $\Delta_{n} = n^{\theta}$ in \eqref{eq:3.3}, we have
\begin{equation*}
\begin{cases}
  \sum_{n=1}^{\infty} \left[n^{(1 - r)(2\theta + 1)} + n^{1 - (1 + \theta)r} \right], & 0 \leqslant \gamma < 1/r \\[5pt]
  \sum_{n=1}^{\infty} \left[n^{(1 - r)(2\theta + 1)} + n^{1 - (1 + \theta)r} \log n \right], & \gamma = 1/r \\[5pt]
  \sum_{n=1}^{\infty} \left[n^{r(\gamma - \theta - 1)} + n^{r(\gamma - 1)(2\theta + 1)} \right], & 1/r < \gamma < 1/2 \\[5pt]
  \sum_{n=1}^{\infty} \left[n^{(r - 1)(2\theta + 1)} \log^{r-1}n + n^{-r(\theta + 1/2)} \right], & \gamma = 1/2 \\[5pt]
  \sum_{n=1}^{\infty} n^{r(\gamma - \theta - 1)}, & 1/2 < \gamma < 1
\end{cases}
\end{equation*}
and hence, the prior five series converge if
\begin{equation}\label{eq:3.16}
	\begin{cases}
		\theta > (2 - r)/(2r - 2), & 0 \leqslant \gamma \leqslant 1/r \\
		\theta > 1/(2r - 2r \gamma) - 1/2, & 1/r < \gamma < 1/2 \\
		\theta > (2 - r)/(2r), & \gamma = 1/2 \\
		\theta > \gamma - 1 + 1/r, & 1/2 < \gamma < 1.
	\end{cases}
\end{equation}
We additionally consider the following simulation parameters: \\

\noindent i. $\sigma=2$, $\mu_{n,k}=1$, $\tau^*=0.5$
	and $\gamma = 0,0.1,0.5,0.7,0.9$; \\[-3pt]	
		
\noindent ii. $\Delta_n=n^{\theta}$ with
\begin{equation*}
\theta =
	\begin{cases}
		-0.29, -0.2, -0.09, 0, 0.1 & \mbox{if } \gamma = 0, 0.1, 0.5 \\
		-0.19, -0.1, -0.01, 0, 0.1 & \mbox{if } \gamma = 0.7 \\
		-0.01, 0, 0.1 & \mbox{if } \gamma = 0.9
	\end{cases}
\end{equation*}
where the values of $\theta$ are in agreement with \eqref{eq:3.16}; \\[-3pt]
		
\noindent iii. $n=50,100,500,1000,1500,2000,3000,4000$; \\[-3pt]
	
\noindent iv. $m=1000$ replications. \\

\noindent The simulation results are displayed in Figures \ref{tab1} and \ref{tab2}. In all cases, and as expected, one observes the convergence of the CUSUM estimator from some value of $n$ onwards.	

\begin{figure}[h!]
\centering
\includegraphics[width=1.0\textwidth]{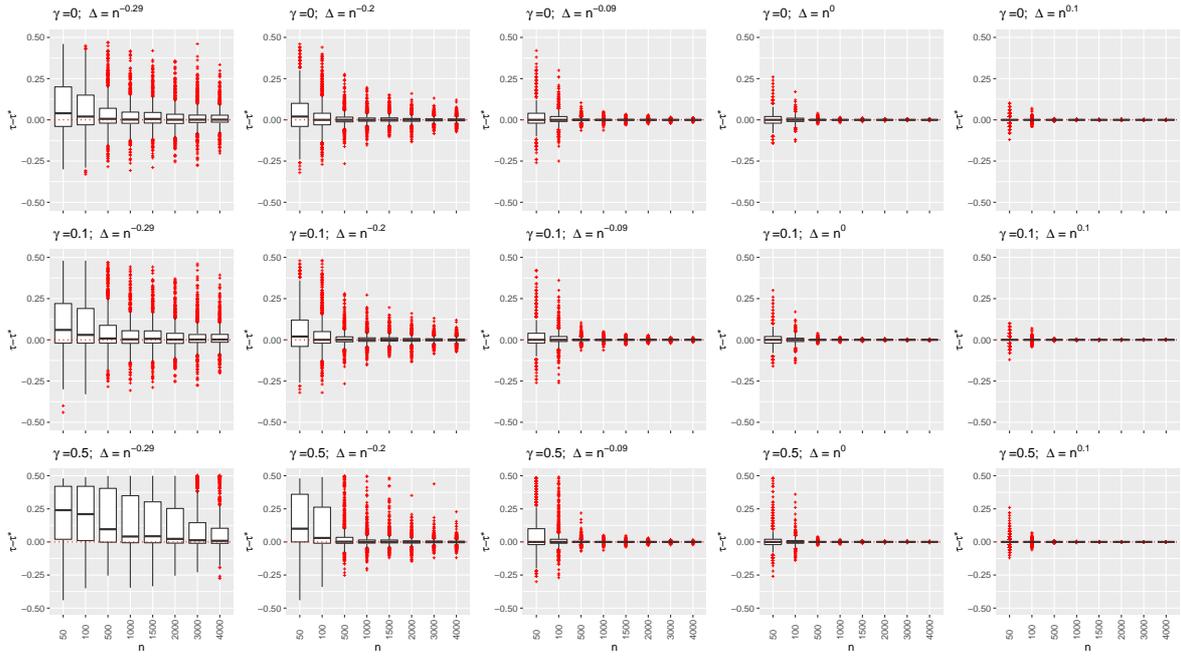}	
\caption{Simulation boxplots referring to parameters $\gamma=0,0.1,0.5$}
\label{tab1}
\end{figure}

\medskip

\begin{figure}[h!]
\centering
\includegraphics[width=1.0\textwidth]{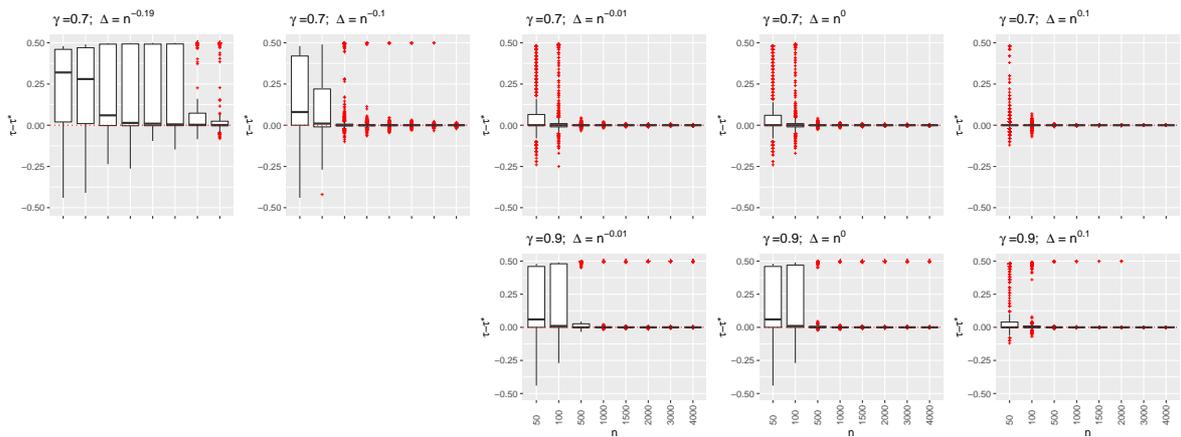}	
\caption{Simulation boxplots referring to parameters $\gamma=0.7,0.9$}
\label{tab2}
\end{figure}

\subsection*{Supplementary materials}

The \texttt{R} simulation codes are available upon request.

\section*{Acknowledgements}

This work is funded by national funds through the FCT - Fundação para a Ciência e a Tecnologia, I.P., under the scope of projects UIDB/04035/2020 (GeoBioTec) and UIDB/00297/2020 (Center for Mathematics and Applications)

\end{document}